\documentclass[11pt]{amsart}

\usepackage{amsfonts, amsmath, amssymb, stmaryrd, color, enumerate, relsize, mathrsfs,array}
\usepackage[pdftex]{graphicx}
\usepackage{xy}
\usepackage{mathdots}
\usepackage{ragged2e} 
\usepackage[utf8]{inputenc}
\usepackage[english]{babel}
\usepackage{geometry}
\usepackage{mathtools}
\usepackage{amsthm}
\usepackage{graphicx}
\usepackage{bbm}
\usepackage{multirow}
\usepackage{multicol}
\usepackage{subfigure}
\usepackage{float}
\usepackage{enumitem}
\usepackage{cancel}
\usepackage{hyperref}
\usepackage{tikz-cd}
\usepackage{float}

\title[Acyclicity of Broussous-Schneider Coefficient Systems]{Acyclicity of Broussous-Schneider\\ Coefficient Systems}
\author[J. Navarro]{Javier Navarro}

\usetikzlibrary{calc,patterns,angles,quotes}
\hypersetup{
    colorlinks=false,
    breaklinks=true
}

\newcommand{\conj}[1]{\left\lbrace#1\right\rbrace}

\newtheorem*{thm*}{Theorem}

\newtheorem*{lem*}{Lemma}

\theoremstyle{definition}

\theoremstyle{remark}
\newtheorem{rem}{Remark}[section]
\theoremstyle{remark}

\DeclareMathOperator{\GL}{GL}
\DeclareMathOperator{\Her}{Her}
\DeclareMathOperator{\Res}{Res}

\DeclareMathOperator{\Hom}{Hom}
\DeclareMathOperator{\Ind}{Ind}
\DeclareMathOperator{\cInd}{c-Ind}
\DeclareMathOperator{\Out}{Out}

\makeatletter
\def\namedlabel#1#2{\begingroup
    #2%
    \def\@currentlabel{#2}%
    \phantomsection\label{#1}\endgroup
}
\makeatother

\let\olditemize\enumerate
\def\enumerate{\olditemize\itemsep=0pt}

\newenvironment{Theorem}[1][]{\par\noindent
   \textbf{Theorem.}\space\emph{#1}}{}

\newenvironment{Proposition}[1][]{\par\noindent
   \textbf{Proposition.}\space#1}{}  
   
\newenvironment{Lemma}[1][]{\par\noindent
   \textbf{Lemma.}\space#1}{}

\tikzset{
  symbol/.style={
    draw=none,
    every to/.append style={
      edge node={node [sloped, allow upside down, auto=false]{$#1$}}}
  }
}
\date{\today}

\begin{document}

\begin{abstract}
     We give a proof of the Acyclicity Conjecture stated by Broussous and Schneider in \cite{BS2017type}. As a consequence, we obtain an exact resolution of every admissible representation on each Bernstein block of $\GL(N)$ associated to a simple type. 
\end{abstract}    
\maketitle

\section{Introduction}
    Let $F$ be a non-Archimedean local field and let $G$ be the group $\GL_N(F)$ of $F$-points of the general linear group $\GL_N$, where $N\geq 2$ is a fixed integer. In this work, we build on the results of Broussous and Schneider \cite{BS2017type} and prove the exactness of the augmented chain complex \[
    C_{\bullet}^{\rm or}(X_\pi,\mathcal{C}[\pi])\longrightarrow\mathcal{V}
    \] 
    for all admissible smooth representation $(\pi,\mathcal{V})$ of $G$ in each Bernstein block corresponding to a simple type. Here, $\mathcal{C}[\pi]$ denotes the Broussous-Schneider coefficient system on the simplicial complex $X_\pi$ constructed in [\emph{loc.\,cit.}] (see also \S\,\ref{sec:BScoeff}).
    
    To explain our results in more detail, let $\mathcal{R}(G)$ be the category of smooth complex representations of $G$. The Bernstein decomposition \cite{bernstein84} states that $\mathcal{R}(G)$ decomposes into a product of subcategories $\mathcal{R}^{\mathfrak{s}}(G)$ indexed by the set of inertial equivalence classes $\mathfrak{I}$ consisting of pairs $(L,\sigma)$ of a Levi subgroup $L$ of $G$ together with an irreducible supercuspidal representation $\sigma$ of $L$:
\[
\mathcal{R}(G)=\prod_{\mathfrak{s}\in\mathfrak{I}}\mathcal{R}^\mathfrak{s}(G).
\]
This leads to the insightful approach to studying $\mathcal{R}(G)$ by restricting to the supercuspidal representations of Levi subgroups of $G$. 

The fundamental results of Bernstein interact with the theory of \emph{types}. The latter approach is through restriction to compact open subgroups, where we base ourselves in Bushnell-Kutzko \cite{bushnell1998smooth}. Specifically, a pair $(J,\lambda)$, where $J$ is a compact open subgroup of $G$ and $\lambda$ is an irreducible smooth representation of $\lambda$, is called an $\mathfrak{s}$-type if the irreducible smooth representations in $\mathcal{R}^\mathfrak{s}(G)$ are precisely those that contain $\lambda$ upon restriction to $J$. If $(J,\lambda)$ is an $\mathfrak{s}$-type, then $\mathcal{R}^\mathfrak{s}(G)$ coincides with the category $\mathcal{R}_{(J,\lambda)}(G)$ of smooth representations that are generated by their $\lambda$-isotypic component. The existence of $\mathfrak{s}$-types for all $\mathfrak{s}\in\mathfrak{I}$ was established by Bushnell and Kutzko \cite{bushnell1993admissible, bushnell1999semisimple}, who also provided a systematic method for constructing them.

To illustrate the interplay between types and coefficient systems, let us consider the case where $\mathfrak{s}$ is the inertial equivalence class of a pair $(G,\pi)$, with $\pi$ an irreducible supercuspidal representation of $G$. Let $(J,\lambda)$ be the $\mathfrak{s}$-type constructed in \cite{bushnell1993admissible}. By the theory in [\emph{loc.\,cit.}], there exists a finite field extension $E/F$ and an embedding $E^\times\hookrightarrow G$ such that $\pi$ is compactly induced: \[
\pi \cong \cInd_{E^\times J}^G\Lambda, 
\]
where $\Lambda$ is an extension of $\lambda$ to $E^\times J$ (Theorem 6.2.2 [\emph{loc.\,cit.}]). The group $E^\times J$ is contained in the stabilizer $G_x$ of a point $x$ in the Bruhat-Tits building $X$ of $G$. Here, $x$ is the barycenter of a facet in $X$ and corresponds to a vertex of the building $X_E$ of $\Res_{E/F}\GL_{N/[E:F]}$ over $F$ under an affine embedding $X_E\hookrightarrow X$. By defining $\overline{\Lambda}= \cInd_{E^\times J}^{G_x}\Lambda$, we obtain an isomorphism
\begin{equation}\label{eq:intr:1}
    \pi \cong \cInd_{G_x}^G\overline{\Lambda}.
\end{equation}
Let $X_\pi$ be the $0$-dimensional simplicial complex given by the orbit $G\cdot x$ in $X$, and define the $G$-equivariant coefficient system $\mathcal{C}[\pi]$ on $X_\pi$ by assigning to $x$ the $\overline{\Lambda}$-isotypic space $\mathcal{V}^{\overline{\Lambda}}$ (and to $g\cdot x$ the space $g\mathcal{V}^{\overline{\Lambda}}\subset \mathcal{V}$). From \eqref{eq:intr:1}, it follows that there is an isomorphism of $G$-spaces:
\begin{equation*}
C_0^{\rm or}(X_\pi, \mathcal{C}[\pi])\xrightarrow{\ \sim\ } \mathcal{V},
\end{equation*}
where $C_0^{\rm or}(X_\pi, \mathcal{C}[\pi])$ denotes the space of (oriented) $0$-chains of $X_\pi$ with coefficients in $\mathcal{C}[\pi]$. (The superscript ``or" is purely notational in this case.) 

This geometric realization of representations was first introduced by Schneider and Stuhler in \cite{schneider1993resolutions} for general linear groups, and later extended to arbitrary connected reductive groups in \cite{schneider1997representation}. Broussous and Schneider further refined it in \cite{schneider1993resolutions} for Bernstein blocks associated to \emph{simple types} of $G$. Here, an $\mathfrak{s}$-type is said to be simple if $\mathfrak{s}$ has a representative $(L,\sigma)$ where $L$ is of the form $\GL_{N/e}(F)\times\cdots\times \GL_{N/e}(F)$ for a positive integer $e$ dividing $N$, and $\sigma\cong \sigma_0\otimes\cdots\otimes\sigma_0$ for an irreducible supercuspidal representation $\sigma_0$ of $\GL_{N/e}(F)$. They proved that the corresponding augmented chain complex is exact when $\pi$ is a twist of an irreducible discrete series representation and reduced the general case to a technical conjecture (Conjeture X.4.1 [\emph{loc.\,cit.}]). As an important application, this yielded explicit pseudo-coefficients and character formulas for discrete series representations. 

In this paper, we prove Broussous-Schneider's conjecture, establishing the exactness of the chain complex for all admissible $\pi$ (see Theorem \ref{main:thm} below). Our proof follows the method of \cite{BS2017type}, with a key new ingredient being a result on the representation theory of parahoric subgroups stated in Lemma \ref{lem:techn}, which, in the discrete series case, was established in Lemma XI.2.3 [\emph{loc.\,cit.}]. 

Recently, in \cite{dat2024}, Dat has constructed coefficient systems attached to smooth representations of any connected reductive group that splits over a tamely ramified extension of $F$, under mild assumptions on $p$, following the Kim-Yu construction of types \cite{kim2017construction}. In \S,\ref{ss:dat}, we observe that this construction coincides with that of Broussous and Schneider when $p\nmid N$.

The paper is organized as follows. After fixing notation in \S\,\ref{sec:notation}, we establish a result on the structure of representations of parahoric subgroups in \S\,\ref{sec:tech:lemma}. In \S\,\ref{sec:BScoeff}, we recall the construction of coefficient systems by Broussous and Schneider. Finally, in \S\,\ref{sec:main}, we prove our main result. 

\subsection*{Acknowledgments} The author thanks his advisors, Anne-Marie Aubert and Luis Lomel\'i for their guidance and support  throughout this work, as well as for their valuable feedback on earlier versions of this paper. He is also grateful to Paul Broussous for insightful discussions on the subject, and to Guy Henniart for his punctual observations on a previous draft. The author acknowledges the hospitality and stimulating environment provided by the Instituto de Matem\'aticas of the Pontificia Universidad Cat\'olica de Valpara\'iso and the Institut de Math\'ematiques de Jussieu-Paris Rive Gauche. The author was partially supported by FONDECYT Grant 1212013 and ANID National Doctoral Scholarship 21211693. 

\section{Notation}\label{sec:notation}

\subsection{} We work with a non-Archimedean local field $F$ with finite residue field of characteristic $p$. We let $G$ denote the group $\GL_N(F)$ of $F$-points of the general linear group $\GL_N$, for some integer $N\geq 2$. For convenience, we fix an $N$-dimensional $F$-vector space $V$ and identify $G={\rm Aut}_FV$, after choosing a basis. We often work with field extensions $E/F$ of degree dividing $N$ and take an embedding of groups $\iota\colon E^\times\hookrightarrow G$. Since our constructions only involve simple types, the choice of $\iota$ is inconsequential, and we systematically identify $E^\times$ with its image in $G$ and view $E$ as an $F$-subalgebra of $A={\rm End}_FV$, omitting explicit reference to $\iota$. For such $E$, we denote by $G_E$ the subgroup of $G$ given by ${\rm Aut}_EV$.

\subsection{} For every finite field extension $E/F$ we denote by $\mathfrak{o}_E$ the ring of integers of $E$, and write $\mathfrak{p}_E$ for the maximal ideal of $\mathfrak{o}_E$. We let $k_E=\mathfrak{o}_E/\mathfrak{p}_E$, the residue field of $E$. For an hereditary $\mathfrak{o}_F$-order $\mathfrak{A}$ in $A$ with Jacobson radical $\mathfrak{P}$, we let $U(\mathfrak{A})$ denote the multiplicative group $\mathfrak{A}^\times$, and $U^1(\mathfrak{A})= 1+\mathfrak{P}$ the subgroup of 1-units. If $E$ is contained in $A$ and $\mathfrak{B}$ is a hereditary $\mathfrak{o}_E$-order in ${\rm End}_E V$, we denote by $\mathfrak{A}(\mathfrak{B})$ the hereditary $\mathfrak{o}_F$-order in $A$ normalized by $E$ such that $\mathfrak{B}={\rm End}_E V \cap \mathfrak{A}(\mathfrak{B})$.

\subsection{} We denote by $X$, resp. $X_E$, the reduced Bruhat-Tits building of $G$, resp. $G_E$. If $E/F$ is Galois, then $X_E$ is the building of $\Res_{E/F}\GL_{N/[E:F]}$ over $F$. For each $E$, we view $X_E$ as embedded in $X$ via the unique affine and $G_E$-equivariant map $j_E\colon X_E\to X$ (cf., \cite[Theorem 1.1]{broussous2002building}). We shall often make use of an anti-isomorphism between hereditary orders $\Her(A)$ in $A$ and simplices $F(X)$ of $X$:
\[
\Her(A)^{\rm opp} \longrightarrow F(X). 
\]
See \cite[Corollaire 2.15]{bruhat1984schemas} and \cite[\S\,I.2]{BS2017type}.

\subsection{} We let $\mathcal{R}(G)$ be the category of smooth complex representations of $G$. If $\chi$ is a character of the center $Z(G)\cong F^\times$ of $G$, we let $\mathcal{R}_\chi(G)$ denote the full subcategory of $\mathcal{R}(G)$ consisting of representations with central character $\chi$. For any pair of groups $H\leq K$ and a representation $\pi$ of $H$, we define $H^g$ to be the group $gHg^{-1}$, and use $\pi^g$ to denote the representation of $H^g$ given by $\pi^g(ghg^{-1})=\pi(h)$, for all $g\in K$. Also, we denote by $N_K(H)$, resp. $C_K(H)$, the normalizer, resp. centralizer, of $H$ in $K$, and let $Z(K)=Z_K(K)$ be the center of $K$.

\section{Results from hereditary orders and the theory of types}\label{sec:tech:lemma}

The aim of this section is to establish a result,  Lemma \ref{lem:techn} below, concerning the represen\-tation theory of parahoric subgroups. 

\subsection{}\label{ss:her:ord} Let us start by reviewing some well-known results on hereditary orders. We take a field extension $E/F$ in $A$ and a hereditary $\mathfrak{o}_E$-order $\mathfrak{B}$ in $B={\rm End}_E V$. Let $\mathcal{L}=\left\lbrace L_i : i\in\mathbb{Z}\right\rbrace$ be a lattice chain in $V$ such that 
\[
\mathfrak{B}={\rm End}_{\mathfrak{o}_E}^0(\mathcal{L})\coloneqq\left\lbrace x\in B : xL_i \subseteq L_i\hspace{0.5em} \mbox{for all}\hspace{0.5em} i\in\mathbb{Z}\right\rbrace,
\]
(cf., \cite[\S\,1]{bushnell1993admissible}). We let $e=e(\mathfrak{B}\vert \mathfrak{o}_E)$ denote the $\mathfrak{o}_E$-period of $\mathcal{L}$, and let $R=\dim_E(V)$. We choose an $\mathfrak{o}_E$-basis $\left\lbrace v_1,\ldots, v_R\right\rbrace$ of $\mathcal{L}$ such that it generates $L_0$. Let us recall here that this means that $\left\lbrace v_1,\ldots, v_R\right\rbrace$ is an $\mathfrak{o}_E$-basis of $L_0$ such that
\[
L_j= \coprod_{k=1}^R \mathfrak{p}_E^{a(j,k)}v_k, \quad j\in\mathbb{Z},
\]
for integers $a(j,1)\leq \cdots \leq a(j,k)$. 
We regard $G_E$ as $\GL_R(E)$ under this choice of basis. For $0\leq i\leq e-1$, we put $n_i=n_i(\mathfrak{B})=\dim_{k_E}(L_i/L_{i+1})$, and set $n_{-1}=0$. For each $i$, we let $V^i$ be the $E$-linear span of the set $\left\lbrace v_{n_{i-1}+1},\dots, v_{n_i}\right\rbrace$. We write $B^i={\rm End}_E(V^i)$ and $\mathfrak{B}^i = \mathfrak{B}\cap B^i$. In this way, $\mathfrak{B}^i$ is a maximal order in $B^i$ and we can identify $U(\mathfrak{B})/U^1(\mathfrak{B})$ with 
\[
\mathbb{L}_{\mathfrak{B}}(k_E)=\prod_{i=0}^{e-1} (\mathfrak{B}^i/\mathfrak{p}_E\mathfrak{B}^i)^\times = \prod_{i=0}^{e-1} \GL_{n_i}(k_E).
\]
We set $\mathfrak{M}(\mathfrak{B})=\prod \mathfrak{B}^i$, viewed as a subgroup of $U(\mathfrak{B})$. If $\mathfrak{B}$ is maximal, then we write $\mathbb{L}_{\mathfrak{B}}(k_E)=\mathbb{G}(k_E)=\GL_R(k_E)$.

    To facilitate subsequent calculations, we further develop the matrix framework and recall some terminology from \cite{bushnell1987hereditary} and \cite{bushnell1993admissible}. Consider the $e$-tuple $\mathbf{n}=(n_0,\ldots, n_{e-1})$. We denote $\mathbb{M}(\mathbf{n}, \mathfrak{o}_E)$ as the collection of $e\times e$ block matrices over $\mathfrak{o}_E$ where the $(i,j)$-block has dimension $n_i\times n_j$. The elements of $\mathbb{M}(\mathbf{n}, \mathfrak{o}_E)$ are referred to as $\mathbf{n}$-block matrices over $\mathfrak{o}_E$. Each matrix in $\mathbb{M}(\mathbf{n}, \mathfrak{o}_E)$ is represented as $(A_{ij})$, where $A_{ij}$ denotes the $(i,j)$-block matrix. We denote by $\mathbb{M}_{n_i}^{n_j}(\mathfrak{o}_E)$ the ring of $n_i\times n_j$-dimensional matrices over $\mathfrak{o}_E$, and we fix a uniformizer $\pi_E\in\mathfrak{o}_E$. With this, the order $\mathfrak{B}$ is identified with the ring of $\mathbf{n}$-block matrices $(A_{ij})$ over $\mathfrak{o}_E$  as follows
\begin{equation}
    \mathfrak{B}=\left\lbrace (A_{ij}) \vert A_{ij}\in\mathbb{M}_{n_i}^{n_j}(\mathfrak{o}_E) \mbox{ if } i\leq j \mbox{ and } A_{ij}\in \pi_E\mathbb{M}_{n_i}^{n_j}(\mathfrak{o}_E) \mbox{ if } i>j   \right\rbrace,
\end{equation}
and the radical $\mathfrak{Q}$ with the group
 \begin{equation}
    \mathfrak{Q}=\left\lbrace (A_{ij}) \vert A_{ij}\in\mathbb{M}_{n_i}^{n_j}(\mathfrak{o}_E) \mbox{ if } i<j \mbox{ and } A_{ij}\in \pi_E\mathbb{M}_{n_i}^{n_j}(\mathfrak{o}_E) \mbox{ if } i\geq j   \right\rbrace.
\end{equation}

\subsection{} Let us fix now a principal hereditary $\mathfrak{o}_E$-order $\mathfrak{B}_0={\rm End}_{\mathfrak{o}_E}^0(\mathcal{L}_0)$, for a lattice chain $\mathcal{L}_0=\left\lbrace L_i^0 : i\in\mathbb{Z}\right\rbrace$. We set $e_0=e(\mathfrak{B}_0\vert \mathfrak{o}_E)$ and denote $f$ as the common value of the integers $n_i(\mathfrak{B}_0)$ for $0\leq i\leq e_0-1$. We let $\rho$ be an irreducible representation of $U(\mathfrak{B}_0)$ that is the inflation of a cuspidal representation of $U(\mathfrak{B}_0)/U^1(\mathfrak{B}_0)$, also denoted by $\rho$, of the following form: upon identifying $U(\mathfrak{B}_0)/U^1(\mathfrak{B}_0)$ with $\prod_{i=0}^{e_0-1} \GL_{f}(k_E)$, $\rho$ takes the form $\rho_0^{\otimes e_0}$, for an irreducible cuspidal representatioin $\rho_0$ of $\GL_f(k_E)$. A wide range of problems concerning the representation theory of $G$ can be reduced to problems involving the representations of $G_E$ containing representations like $\rho$, primarily through Hecke algebra isomorphisms. One of the most basic facts on which this reduction relies is that the space
\[
I_x(\rho)=\Hom_{U(\mathfrak{B}_0)\cap U(\mathfrak{B}_0)^x}(\rho,\rho^x)
\]
is either zero or one-dimensional. In particular, the dimension of $I_x(\rho)$ remains constant for all $x$ in the intertwining set $I_{G_E}(\rho)=\left\lbrace x\in G_E : I_x(\rho)\not= 0\right\rbrace$ of $\rho$. In Lemma \ref{lem:techn} below we shall see that a similar assertion holds even if we consider representations of $U(\mathfrak{B})$ for hereditary orders $\mathfrak{B}$ containing $\mathfrak{B}_0$.

\subsection{}\label{ss:weyl:intertw} Before delving into the technicalities of the lemma, it is useful to recall some standard facts regarding the intertwining of $\rho$. For details, see \cite{bushnell1993admissible}. Let $L/E$ be an unramified extension of degree $f$ in $B$ such that $\mathfrak{B}_0$ is normalized by $L^\times$. We set $C={\rm End}_L V$ and $\mathfrak{C}_{\rm min}= C\cap \mathfrak{B}_0$; the order $\mathfrak{C}_{\rm min}$ is minimal in $C$. In this manner, $G_L$ is contained in $I_{G_E}(\rho)$ and, furthermore, the inclusion $G_L\to G_E$ gives a bijection
\begin{equation}\label{eq:double:intertwining}
U(\mathfrak{C}_{\rm min})\backslash G_L/ U(\mathfrak{C}_{\rm min}) \longleftrightarrow U(\mathfrak{B}_0)\backslash I_{G_E}(\rho)/ U(\mathfrak{B}_0).
\end{equation}
The above is essentially part of the Hecke algebra isomorphism 
\[
\mathcal{H}(G_L, 1_{U(\mathfrak{C}_{\rm min})})\xrightarrow{\ \sim\ }\mathcal{H}(G_E,\rho)
\]
established in \cite[\S\,5]{bushnell1993admissible}. Let us recall an explicit set of representatives for \eqref{eq:double:intertwining}. We embed the group $\Sigma_{e_0}$ of permutations of the set $\left\lbrace 1,\ldots, e_0\right\rbrace$ into the group $W_0$ of permutations of the basis $\left\lbrace v_1,\ldots, v_n\right\rbrace$ by defining the permutation 
\[
v_{d(i)+j}\longmapsto v_{d(s(i))+j},  \quad 1\leq i\leq e_0,\ 0\leq j\leq f-1,
\]
for each $s\in \Sigma_{e_0}$, where $d(i)=(i-1)f$. We write $W(\mathfrak{B}_0)$ for the image of $\Sigma_{e_0}$ in $W_0$ via this embedding. We define $D(\mathfrak{B}_0)$ as the subgroup of diagonal matrices in $G_E$ whose eigenvalues are powers of $\pi_E$ and which centralize $\mathfrak{M}(\mathfrak{B}_0)^\times$. The group $W(\mathfrak{B}_0)$ normalizes $D(\mathfrak{B}_0)$ and 
\begin{equation}\label{eq:weyl:intertw}
\widetilde{W}(\mathfrak{B}_0)= W(\mathfrak{B}_0) \ltimes D(\mathfrak{B}_0),
\end{equation}
the group generated by $W(\mathfrak{B}_0)$ and $D(\mathfrak{B}_0)$, is a set of representatives of \eqref{eq:double:intertwining}.
\subsection{A useful Lemma}
We now arrive at the lemma. We keep the notation above. In particular, $\mathfrak{B}_0$ is a principal hereditary $\mathfrak{o}_E$-order in $B$ and $\rho$ is an irreducible representation of $U(\mathfrak{B}_0)$ inflated from a cuspidal representation of $U(\mathfrak{B}_0)/U^1(\mathfrak{B}_0)$ of the form $\rho_0^{\otimes e_0}$. Let $\mathfrak{B}={\rm End}_{\mathfrak{o}_e}^0(\mathcal{L})$ be a hereditary $\mathfrak{o}_E$-order in $B$ containing $\mathfrak{B}_0$. Shifting the index of $\mathcal{L}=\left\lbrace L_i : i\in\mathbb{Z}\right\rbrace$ and changing $\left\lbrace v_1,\ldots, v_R\right\rbrace$ if necessary, we may assume that it is a common basis of $\mathcal{L}_0$ and $\mathcal{L}$, and that $L^0_0=L_0$. 

\subsubsection{}\label{lem:techn}\begin{Lemma}
    \emph{Let $\tau$ be the inflation of an irreducible representation of $U(\mathfrak{B})/U^1(\mathfrak{B})$ to $U(\mathfrak{B})$. Then
    \begin{equation}\label{eq:lem:techn}
        \dim \Hom_{U(\mathfrak{B}_0)\cap U(\mathfrak{B})^x}(\rho, \tau ^x) = \dim \Hom_{U(\mathfrak{B}_0)}(\rho, \tau)
    \end{equation}
    for all $x\in I_{G_E}(\rho)$.}
\end{Lemma}

\begin{proof}
    Since $U(\mathfrak{B}_0)$ is contained in $U(\mathfrak{B})$, it is sufficient to verify \eqref{eq:lem:techn} for a set of representatives of $U(\mathfrak{B}_0)\backslash I_{G_E}(\rho)/ U(\mathfrak{B}_0)$. Let us then take $x\in\widetilde{W}(\mathfrak{B}_0)$. We first assume that $x\in D(\mathfrak{B}_0)$ and satisfies the following condition: 
    \begin{itemize}
        \item[\namedlabel{cond:D}{(D)}]\label{D1} If $x={\rm diag}(\pi_E^{a_1},\dots, \pi_E^{a_R})$, then $a_{n_{i-1}+1}\geq\cdots\geq a_{n_{i}}$ for all $0\leq i\leq e-1$. 
    \end{itemize}
    Recall that $x$ centralizes $\mathfrak{M}(\mathfrak{B}_0)^\times$, thus \[
    a_{jf+1}=a_{jf+2}=\cdots = a_{(j+1)f} \quad \mbox{for all } 0\leq j\leq e_0-1.
    \]
    To simplify notation, we write $b_j$ for $a_{jf}$ for all $1\leq j\leq e_0$, and set $d=x^{-1}$. Then we have 
    \[
    \Hom_{U(\mathfrak{B}_0)\cap U(\mathfrak{B})^x}(\rho, \tau ^x)=\Hom_{U(\mathfrak{B}_0)^d\cap U(\mathfrak{B})}(\rho^d, \tau ).
    \]
    Now, since $d$ normalizes $\mathfrak{M}(\mathfrak{B}_0)^\times$, then
\[
\frac{U(\mathfrak{B}_0)^d\cap U(\mathfrak{B})}{U^1(\mathfrak{B}_0)^d\cap U(\mathfrak{B})}=\frac{U(\mathfrak{B}_0)}{U^1(\mathfrak{B}_0)},
\]
and $\rho^d\vert U(\mathfrak{B}_0)^d\cap U(\mathfrak{B})$ is the inflation of the representation $\rho^d$ of $\mathbb{L}_{\mathfrak{B}_0}(k_E)$. But $d$ centralizes $\mathbb{L}_{\mathfrak{B}_0}(k_E)$, hence $\rho^d=\rho$ as $\mathbb{L}_{\mathfrak{B}_0}(k_E)$-modules. Therefore
\[
\Hom_{U(\mathfrak{B}_0)^d\cap U(\mathfrak{B})}(\rho^d, \tau)=\Hom_{\mathbb{L}_{\mathfrak{B}_0}(k_E)}(\rho, \tau_{U^1(\mathfrak{B}_0)^d\cap U(\mathfrak{B})}), 
\]
where $\tau_{U^1(\mathfrak{B}_0)^d\cap U(\mathfrak{B})}$ is the representation of $\mathbb{L}_{\mathfrak{B}_0}(k_E)$ on the space of $U^1(\mathfrak{B}_0)^d\cap U(\mathfrak{B})$-coinvariants of $\tau$. Since $\tau$ is the inflation of a representation of $\mathbb{L}_{\mathfrak{B}}(k_E)$, then \eqref{eq:lem:techn} follows if $U^1(\mathfrak{B}_0)^d\cap U(\mathfrak{B})\subseteq U^1(\mathfrak{B}_0)$ and
\begin{equation}\label{eq:quot:inv:1}
\frac{U^1(\mathfrak{B}_0)^d\cap U(\mathfrak{B})}{U^1(\mathfrak{B}_0)^d\cap U^1(\mathfrak{B})}=\frac{U^1(\mathfrak{B}_0)}{U^1(\mathfrak{B})}.
\end{equation}
Indeed, in this case we have $\tau_{U^1(\mathfrak{B}_0)^d\cap U(\mathfrak{B})}=\tau_{U^1(\mathfrak{B}_0)}$ and
\[
\Hom_{U(\mathfrak{B}_0)^d\cap U(\mathfrak{B})}(\rho^d, \tau)=\Hom_{\mathbb{L}_{\mathfrak{B}_0}(k_E)}(\rho, \tau_{U^1(\mathfrak{B}_0)})=\Hom_{U(\mathfrak{B}_0)}(\rho,\tau).
\]
First, we show that $U^1(\mathfrak{B}_0)^d\cap U(\mathfrak{B})\subseteq U^1(\mathfrak{B}_0)$. By viewing $U(\mathfrak{B})$ and $U^1(\mathfrak{B}_0)$ as $\mathbf{n}$-block matrices over $\mathfrak{o}_E$, we observe that both groups coincide in all $(i,j)$-blocks for $i\not= j$. Consequently, we may restrict to $\mathfrak{M}(\mathfrak{B})^\times$. Moreover, since $d$ is diagonal, it is sufficient to consider any single block, allowing us to assume $\mathfrak{B}$ maximal (note that condition \ref{cond:D} is defined block-wise). Letting $\mathbf{f}$ denote the tuple $(n_0(\mathfrak{B}_0),\ldots, n_{e_0-1}(\mathfrak{B}_0))=(f,\ldots, f)$, the group $U^1(\mathfrak{B}_0)^d$ is represented by the $\mathbf{f}$-block matrices over $\mathfrak{o}_E$ given by
\begin{equation*}
    U^1(\mathfrak{B}_0)^d=\left( \begin{array}{cccccc}
        1+\mathfrak{p}_E & \pi_E^{b_2-b_1}\mathfrak{o}_E & \pi_E^{b_3-b_1}\mathfrak{o}_E & \cdots & \cdots & \pi_E^{b_{e_0}-b_1}\mathfrak{o}_E \\
         \pi_E^{b_1-b_2}\mathfrak{p}_E & 1+\mathfrak{p}_E & \pi_E^{b_3-b_2}\mathfrak{o}_E & \cdots & \cdots & \pi_E^{b_{e_0}-b_2}\mathfrak{o}_E \\ 
         \pi_E^{b_1-b_3}\mathfrak{p}_E & \pi_E^{b_2-b_3}\mathfrak{p}_E & 1+\mathfrak{p}_E & \cdots & \cdots & \pi_E^{b_{e_0}-b_3}\mathfrak{o}_E \\ 
         \vdots & & \ddots & \ddots & & \vdots \\
         \vdots & & & \ddots & \ddots & \vdots \\
         \pi_E^{b_1-b_{e_0}}\mathfrak{p}_E & \cdots & \cdots & \cdots & \pi_E^{b_{e_0-1}-b_{e_0}}\mathfrak{p}_E & 1+\mathfrak{p}_E
    \end{array}
    \right).
\end{equation*}
By condition \ref{cond:D}, $b_j-b_i\leq 0$ for all $i<j$, leading to 
\begin{equation}\label{eq:block:d:lem}
    U^1(\mathfrak{B}_0)^d\cap U(\mathfrak{B})= \left( \begin{array}{cccccc}
        1+\mathfrak{p}_E & \mathfrak{o}_E & \mathfrak{o}_E & \cdots & \cdots & \mathfrak{o}_E \\
         \pi_E^{b_{1,2}}\mathfrak{p}_E & 1+\mathfrak{p}_E & \mathfrak{o}_E & \cdots & \cdots & \mathfrak{o}_E \\ 
         \pi_E^{b_{1,3}}\mathfrak{p}_E & \pi_E^{b_{2,3}}\mathfrak{p}_E & 1+\mathfrak{p}_E & \cdots & \cdots & \mathfrak{o}_E \\ 
         \vdots & & \ddots & \ddots & & \vdots \\
         \vdots & & & \ddots & \ddots & \vdots \\
         \pi_E^{b_{1,e_0}}\mathfrak{p}_E & \cdots & \cdots & \cdots & \pi_E^{b_{e_0-1,e_0}}\mathfrak{p}_E & 1+\mathfrak{p}_E
    \end{array}
    \right),
    \end{equation}
where we abbreviate $b_{i,j}=b_i-b_j$. It follows that $U^1(\mathfrak{B}_0)^d\cap U(\mathfrak{B})\subseteq U^1(\mathfrak{B}_0)$. Let us prove now \eqref{eq:quot:inv:1}. Proceeding block-wise, it again suffices to assume $\mathfrak{B}$ maximal. But in this case, from \eqref{eq:block:d:lem}, the desired identity becomes clear. In fact, both terms in \eqref{eq:quot:inv:1} are identified with the $\mathbf{f}$-block matrices over $k_E$ such that the $(i,j)$-block is zero for all $i\geq j$.  

Now, take any $x\in \widetilde{W}(\mathfrak{B}_0)$, say $x=wd$ for $w\in W(\mathfrak{B}_0)$ and $d\in D(\mathfrak{B}_0)$.  Write $d={\rm diag}(\pi_E^{a_1},\ldots, \pi_E^{a_R})$. As noted previously, we have
    \[
    \Hom_{U(\mathfrak{B}_0)\cap U(\mathfrak{B})^{wd}}(\rho, \tau ^{wd})=\Hom_{U(\mathfrak{B}_0)^{w^{-1}}\cap U(\mathfrak{B})^d}(\rho^{w^{-1}}, \tau ^d).
    \]
Moreover, since both $w$ and $d$ normalize $\mathfrak{M}(\mathfrak{B}_0)^\times$, it follows that
\[
\frac{U(\mathfrak{B}_0)^{w^{-1}}\cap U(\mathfrak{B})^d}{U^1(\mathfrak{B}_0)^{w^{-1}}\cap U(\mathfrak{B})^d}=\frac{U(\mathfrak{B}_0)}{U^1(\mathfrak{B}_0)},
\]
and $\rho^{w^{-1}}\vert U(\mathfrak{B}_0)^{w^{-1}}\cap U(\mathfrak{B})^d$ corresponds to the inflation of the representation $\rho^{w^{-1}}$ of $\mathbb{L}_{\mathfrak{B}_0}(k_E)$. The representation $\rho$ of $\mathbb{L}_{\mathfrak{B}_0}(k_E)$ being $\rho_0^{\otimes e_0}$, we have that $\rho^{w^{-1}}$ is $\mathbb{L}_{\mathfrak{B}_0}(k_E)$-equivalent to $\rho$. It then follows that
\[
\Hom_{U(\mathfrak{B}_0)^{w^{-1}}\cap U(\mathfrak{B})^d}(\rho^{w^{-1}}, \tau ^d)=\Hom_{\mathbb{L}_{\mathfrak{B}_0}(k_E)}(\rho, (\tau ^d)_{U^1(\mathfrak{B}_0)^{w^{-1}}\cap U(\mathfrak{B})^d}).
\]
Upon choosing an isomorphism $\pi_E^{a_i-a_j}\mathfrak{o}_E/\pi_E^{a_i-a_j}\mathfrak{p}_E\xrightarrow{\sim} k_E$ for each $i,j$, we identify the group $U(\mathfrak{B})^d/U^1(\mathfrak{B})^d$ with $\mathbb{L}_{\mathfrak{B}}(k_E)$. Accordingly, $\tau^d$ corresponds with the inflation of the representation $\tau$ of $\mathbb{L}_{\mathfrak{B}}(k_E)$ to $U(\mathfrak{B})^d$. Then $(\tau^d)_{{U^1(\mathfrak{B}_0)^{w^{-1}}\cap U(\mathfrak{B})^d}}$ is the space of $\mathbb{U}(w^{-1},d)$-coinvariants of $\tau^d$, where 
\[
\mathbb{U}(w^{-1},d)=\frac{U^1(\mathfrak{B}_0)^{w^{-1}}\cap U(\mathfrak{B})^d}{U^1(\mathfrak{B}_0)^{w^{-1}}\cap U^1(\mathfrak{B})^d}.
\]
Let us denote by $\overline{w}^{-1}$ the image of $w^{-1}$ under the natural projection 
\[
 N_{\mathbb{G}(k_E)}(\mathbb{L}_{\mathfrak{B}_0}(k_E))/C_{\mathbb{G}(k_E)}(Z(\mathbb{L}_{\mathfrak{B}_0}(k_E)))\longrightarrow N_{\mathbb{G}(k_E)}(\mathbb{L}_{\mathfrak{B}}(k_E))/C_{\mathbb{G}(k_E)}(Z(\mathbb{L}_{\mathfrak{B}}(k_E))).
\]
It follows readily that
\[
\mathbb{U}(w^{-1},d)=\mathbb{U}(1,d)^{\overline{w}^{-1}}=\left(\frac{U^1(\mathfrak{B}_0)\cap U(\mathfrak{B})^d}{U^1(\mathfrak{B}_0)\cap U^1(\mathfrak{B})^d}\right)^{\overline{w}^{-1}}.
\]
By the invariance of $\rho$ by $\overline{w}^{-1}$, we have
\[
\Hom_{\mathbb{L}_{\mathfrak{B}_0}(k_E)}(\rho, (\tau^d)_{\mathbb{U}(1,d)^{\overline{w}^{-1}}})=\Hom_{\mathbb{L}_{\mathfrak{B}_0}(k_E)}(\rho, (\tau^d)_{\mathbb{U}(1,d)})=\Hom_{U(\mathfrak{B}_0)\cap U(\mathfrak{B})^d}(\rho, \tau ^d).
\]
Then, we may assume that $x\in D(\mathfrak{B}_0)$. In addition, we have
\[
\Hom_{U(\mathfrak{B}_0)\cap U(\mathfrak{B})^x}(\rho, \tau ^x)=\Hom_{U(\mathfrak{B}_0)\cap U(\mathfrak{B})^{xw}}(\rho, \tau ^{xw})
\]
for all $w\in W(\mathfrak{B}_0)\cap U(\mathfrak{B})$. This allows us to replace $x$ with its conjugate under any $w\in W(\mathfrak{B}_0)\cap U(\mathfrak{B})$. Consequently, we can assume that $x$ satisfies condition \ref{cond:D}, which completes the proof of the lemma.  
\end{proof}

\begin{rem}
    In the particular case where $\tau$ is the inflation of the generalized Steinberg representation ${\rm St}(\mathfrak{B},\rho)$ of $\mathbb{L}_{\mathfrak{B}}(k_E)$ with cuspidal support $(\mathbb{L}_{\mathfrak{B}_0},\rho)$, the result above is established in \cite[Lemma XI.2.3]{BS2017type}. 
\end{rem}

\section{Broussous-Schneider Coefficient Sysyems}\label{sec:BScoeff}

In this section, we review the Broussous-Schneider coefficient systems and establish the acyclicity of the derived chain complexes for admissible representations. This construction provides a geometric realization of Bushnell-Kutzko's simple types, and we begin with a brief overview of these types. 

\subsection{}Let us fix a simple stratum $[\mathfrak{A}_0,n,0,\beta]$ in $A$ (cf. \cite[\S\,1]{bushnell1993admissible}), where $\beta$ is a certain element in $A$ such that $E=F[\beta]$ is a field and $\mathfrak{A}_0=\mathfrak{A}(\mathfrak{B}_0)$ is a principal $\mathfrak{o}_F$-order in $A$ normalized by $E$. For what follows, we view $E$ as an $F$-subalgebra of $A$. We also fix a continuous character $\psi_F$ of $F$ which is null on $\mathfrak{p}_F$ but not on $\mathfrak{o}_F$. For any $\mathfrak{o}_F$-order $\mathfrak{A}$ in $A$ normalized by $E$, Bushnell-Kutzko's theory of types produces a pair of open compact subgroups $H^1(\beta, \mathfrak{A})\subseteq J^1(\beta, \mathfrak{A})$ of $G$, and a set $\mathcal{C}(\mathfrak{A},0,\beta)$ of characters of $H^1(\beta,\mathfrak{A})$, called simple characters. We note that if $E=F$, then $H^1(\beta,\mathfrak{A})=J^1(\beta,\mathfrak{A})=U^1(\mathfrak{A})$, and $\mathcal{C}(\mathfrak{A},0,\beta)$ consists only of the trivial character. Since the element $\beta$ will be fixed throughout, we often omit it from notation, writing simply $H^1(\mathfrak{A})$ and $J^1(\mathfrak{A})$. Moreover, as we mainly work with hereditary orders normalized by $E$ in $A$, we denote $H^1(\mathfrak{B})$ and $J^1(\mathfrak{B})$ in place of $H^1(\mathfrak{A}(\mathfrak{B}))$ and $J^1(\mathfrak{A}(\mathfrak{B}))$, when there is no ambiguity. 

By the theory of Heisenberg representations, given $\theta\in\mathcal{C}(\mathfrak{A}(\mathfrak{B}),0,\beta)$, there is a unique irreducible representation $\eta(\theta)$ of $J^1(\mathfrak{B})$ containing $\theta$ (cf., \cite[\S\,5.1]{bushnell1993admissible}). Moreover, if $J(\mathfrak{B})=U(\mathfrak{B})J^1(\mathfrak{B})$, then there is an irreducible representation $\kappa$ of $J(\mathfrak{B})$ such that $\kappa\mid J^1(\mathfrak{B})= \eta(\theta)$ and it is intertwined by $B^\times$. Representations with these properties are called $\beta$-extensions. Although not unique, two $\beta$-extensions differ only by a character of the form $\chi\circ \det_B$, for some character $\chi$ of $U(\mathfrak{o}_E)/U^1(\mathfrak{o}_E)$. (See \cite[\S\,5.2]{bushnell1993admissible}.)

A simple type $(J,\lambda)$ is constructed as follows: take $J=J(\mathfrak{B}_0)$, let $\kappa$ be a $\beta$-extension of a simple character $\theta_0\in \mathcal{C}(\mathfrak{A}_0,0,\beta)$, and define $\lambda=\kappa\otimes\rho$, where $\rho$ is the inflation to $J$ of $(\rho_0)^{\otimes e_0}$, with $\rho_0$ an irreducible cuspidal representation of $\GL_f(k_E)$. 

\subsection{} We fix a simple character $\theta_0\in \mathcal{C}(\mathfrak{A}_0,0,\beta)$, a $\beta$-extension $\kappa$, and a simple type $(J,\lambda=\kappa\otimes\rho)$ containing $\kappa$. A notable property of the sets $\mathcal{C}(\mathfrak{A},0,\beta)$ is the \emph{transfer property}, which allows us to interact with various $\mathfrak{o}_F$-orders and even various $E$-vector spaces. To be more specific, let $V_1,$ $V_2$ be finite-dimensional $E$-vector spaces, and let $\mathfrak{B}_i$ be a hereditary $\mathfrak{o}_E$-order in $B_i={\rm End}_{E}V_i$, $i=1,2$. Let $\mathfrak{A}_i$ be the hereditary $\mathfrak{o}_F$-order in $A_i={\rm End}_FV_i$ such that $\mathfrak{B}_i=B_i\cap \mathfrak{A}_i$. Then there exists a canonical bijection (cf. \cite[\S\,3.6]{bushnell1993admissible})
\[
\tau_{\mathfrak{A}_1,\mathfrak{A}_2,\beta}\colon \mathcal{C}(\mathfrak{A}_1,0,\beta)\longrightarrow \mathcal{C}(\mathfrak{A}_2,0,\beta).
\]
In the case where $V_1=V_2$, this bijection is characterized by the property that $\theta$ and $\tau_{\mathfrak{A}_1,\mathfrak{A}_2,\beta}(\theta)$ coincide in $H^1(\mathfrak{A}_1)\cap H^1(\mathfrak{A}_2)$ for all $\theta\in\mathcal{C}(\mathfrak{A}_1,0,\beta)$. With this in hand, we have a map $\Theta_0$ from the set of hereditary $\mathfrak{o}_F$-orders normalized by $E$ to simple characters, by setting 
\[
\Theta_0(\mathfrak{A}(\mathfrak{B}))=\Theta_0(\mathfrak{B})=\tau_{\mathfrak{A}_0,\mathfrak{A},\beta}(\theta_0).
\]
This map is called a potencial simple character, for short ps-character. Note that the ps-character $\Theta_0$ takes hereditary $\mathfrak{o}_F$-orders normalized by $E$ in any finite-dimensional $E$-vector space. In this generality, for each order $\mathfrak{A}(\mathfrak{B})$, we let $\eta(\mathfrak{B})=\eta(\Theta_0,\mathfrak{B})$ denote the unique irreducible representation of $J^1(\mathfrak{B})$ containing $\Theta_0(\mathfrak{B})$.

\subsection{} Returning to our fixed vector space $V$, we consider the following subset of the Bruhat-Tits building $X$ of $G$:
\[
X[E]\coloneqq \bigcup_{g\in G} g X_E.
\]
Recall that the building $X_E$ of $G_E$ is viewed as a subset of $X$ via the unique affine $G_E$-equivariant map $X_E\to X$. The set $X[E]$ is endowed with the simplicial structure where the simplices are the subsets of the form $g\tilde\sigma$ for $\tilde\sigma$ a simplex in $X_E$ and $g\in G$. 

Now, we consider a representation $\mathcal{V}$ in the single Bernstein block $\mathcal{R}_{(J,\lambda)}(G)$ associated with $(J,\lambda)$, which we will henceforth assume to be admissible. For any simplex $\sigma$ in $X[E]$, Broussous-Schneider define the vector space $\mathcal{V}[\sigma]$ as follows. If $\sigma=h\tilde\sigma$, where $\tilde\sigma$ is a simplex in $X_E$ corresponding to a hereditary $\mathfrak{o}_E$-order $\mathfrak{B}$ in $B$, then
\begin{equation}\label{eq:coefsys}
\mathcal{V}[\sigma] \coloneqq \sum_{g\in U(\mathfrak{A}(\mathfrak{B}))/U^1(\mathfrak{A}(\mathfrak{B}))}hg\mathcal{V}^{\eta(\mathfrak{A}(\mathfrak{B}))},
\end{equation}
where $\eta(\mathfrak{A}(\mathfrak{B}))=\Ind_{J^1(\mathfrak{B})}^{U^1(\mathfrak{A})}\eta(\mathfrak{B})$. If $\sigma$ and $\tau$ are simplices in $X[E]$ such that $\sigma\subseteq\tau$, then $\mathcal{V}[\tau]\subseteq\mathcal{V}[\sigma]$. As a result, by taking inclusions as transition maps, the map $\sigma\mapsto \mathcal{V}[\sigma]$ gives rise to a coefficient system $\mathcal{C}[\Theta_0, V, \mathcal{V}]$ on $X[E]$. 

To describe the support of $\mathcal{C}[\Theta_0, V, \mathcal{V}]$, let $L/E$ be the unramified extension of degree $f$ in $B$ that normalizes $\mathfrak{B}_0$, as in \ref{ss:weyl:intertw}. In the same way as above, we define the complex 
\[
X[L]\coloneqq \bigcup_{g\in G} g X_L\subseteq X.
\]
Since $L/E$ is unramified, the inclusion $X_L\to X_E$, and consequently $X[L]\to X[E]$, is simplicial. The support of $\mathcal{C}[\Theta_0, V, \mathcal{V}]$ is precisely $X[L]$, i.e., for any simplex $\sigma$ in $X[E]$, we have $\mathcal{V}[\sigma]\not=0$ if and only if $\sigma\in X[L]$ (cf. \cite[Proposition VIII.2.6]{BS2017type}). We therefore view $\mathcal{C}[\Theta_0, V, \mathcal{V}]$ as a coefficient system on $X[L]$. 

With the representation $\mathcal{V}$ fixed, we can construct coefficient systems as above for another ps-characters $\Theta'$. An interesting aspect here is that this construction localizes, in some sense, $\Theta_0$ up to endo-equivalence. To be more precise, let $[\mathfrak{A}',0,\beta']$ be an arbitrary simple stratum in $A$ and take $\theta'\in\mathcal{C}(\mathfrak{A}',0,\beta')$. Let $E'=F[\beta']$, and let $V'$ denote the space $V$ regarded as an $E'$-vector space. The character $\theta'$ defines a ps-character $\Theta'$, which produces a simple character $\Theta'(\mathfrak{B}')\in\mathcal{C}(\mathfrak{A}(\mathfrak{B}'),0,\beta')$ for each hereditary $\mathfrak{o}_{E'}$-order in $B'={\rm End}_{E'}V'$. In this way, the corresponding coefficient system $\mathcal{C}[\Theta', V', \mathcal{V}]$ is non-zero if and only if $\Theta_0$ and $\Theta'$ are endo-equivalent; that is, if and only if $[E:F]=[E':F]$ and the characters $\Theta_0(\mathfrak{B})$ and $\Theta'(\mathfrak{B}')$ intertwine in $G$ for all $\mathfrak{B}$ and $\mathfrak{B}'$. Furthermore, in this case, we have $X[E]=X[E']$ and $\mathcal{C}[\Theta_0, V, \mathcal{V}]=\mathcal{C}[\Theta', V', \mathcal{V}]$ (cf. \cite[Proposition VIII.1.2]{BS2017type}).  

\subsection{Relation to Dat's Coefficient Systems}\label{ss:dat} Before proceeding further, we note a relationship between $\mathcal{C}$ and the coefficient system constructed by Dat in the ongoing work \cite{dat2024}. In [\emph{loc.\,cit.}], a coefficient system is constructed for every smooth representation of any connected reductive group which is tamely ramified over $F$, with certain additional conditions on the residual characteristic $p$ of $F$ (see \S\,2.3 hypothesis (H1) and \S\,2.4.5 hypothesis (H2), [\emph{loc.\,cit.}]). As we shall see, this construction reproduces the Broussous-Schneider coefficient systems $\mathcal{C}$, in the context of Bernstein blocks associated with simple types of $\GL_N$. We note that the latter does not require any tame assumption on $p$.

Let $\mathbf{G}$ be a connected reductive group defined over $F$ that splits over a tamely ramified extension $E/F$, and let $G={\bf G}(F)$. Denote by $\hat{\bf G}$ the Langlands dual of $\bf G$ over $\mathbb{C}$ and by $^L{\bf G}=\hat{\bf G}\rtimes W_F$ the Weil form of the $L$-group of ${\bf G}$. Let $W_F'$ be the Weil-Deligne group of $F$ and $\Phi(G)\subset H^1(W_F',\hat{\bf G})$ be the set of admissible Langlands parameters (see \cite[\S\,10]{kottwitz1984}). 

The framework of \cite{dat2024} lies in the expectation of a decomposition
\begin{equation*}
\mathcal{R}(G)=\prod_{\phi\in\Phi(P_F,G)}\mathcal{R}^{\phi}(G),
\end{equation*}
where $P_F$ denotes the wild inertial subgroup of $W_F$ and $\Phi(P_F,G)$ is the image of $\Phi(G)$ by the restriction map $H^1(W_F',\hat{\bf G})\to H^1(P_F,\hat{\bf G})$. The category $\mathcal{R}^{\phi}(G)$ is anticipated to be a Serre subcategory where the simple objects are the irreducible representations whose Langlands parameter $\varphi_\pi$ satisfies $\varphi_{\pi\vert P_F}\sim \phi$.
 
Dat's approach is as follows. For $\phi\colon P_F\to{}^L{\bf G}\in\Phi(P_F,G)$, let $Z(C_{\hat{\bf G}}(\phi))^\circ$ be the connected center of the centralizer of $\phi(P_F)$ in $\hat{\bf G}$. Then $\hat{\bf L}_{\phi}\coloneqq C_{\hat{\bf G}}(Z(C_{\hat{\bf G}}(\phi))^\circ)$ is a Levi subgroup of $\hat{\bf G}$. If $\varphi\colon W_F\to{}^L{\bf G}$ extends $\phi$, then the conjugation action ${\rm Ad}_\varphi$ of $W_F$ on $C_{\hat{\bf G}}(\phi)$ preserves its connected center and therefore also $\hat{\bf L}_{\phi}$. The outer action ${\rm Ad}_\varphi\colon W_F\to {\rm Out}(\hat{\bf L}_{\phi})$ is finite and independent of the choice of $\varphi$, hence there exists a quasi-split group ${\bf L}_{\phi}$ over $F$ endowed with a $W_F$-equivariant isomorphism of based root data $\psi_0({\bf L}_\phi)\xrightarrow{\ \sim\ }\psi_0(\hat{\bf L}_{\phi})^{\vee}$. Let $\hat{\bf L}_{\phi,{\rm ab}}$ be the cocenter of $\hat{\bf L}_{\phi}$ and denote by ${\bf S}_\phi$ the $F$-torus dual to the complex torus $\hat{\bf L}_{\phi,{\rm ab}}$ (with its $W_F$-action). The embedding $Z(\hat{\bf L}_{\phi})^\vee\subset \hat{\bf G}$ gives rise to a dual ${\bf G}$-conjugacy class of embeddings ${\bf S}_\phi\hookrightarrow {\bf G}$ (cf., \cite{dat2024}), which we denote by $I_\phi$. According to [\emph{loc.\,cit.}], there is an $F$-rational embedding $\iota\in I_\phi$ such that ${\bf G}_{\iota}\coloneqq C_{\bf G}(\iota({\bf S}_\phi))$ is naturally isomorphic to ${\bf L}_\phi$. 

If ${\bf S}$ is a maximal $F$-torus of ${\bf G}$ that splits over some tamely ramified finite extension $E$ of $F$, we put $\mathcal{A}({\bf G}, {\bf S}, F)\coloneqq\mathcal{A}({\bf G}, {\bf S}, E)\cap \mathcal{B}({\bf G}, F)$, where $\mathcal{A}({\bf G}, {\bf S}, E)\cap \mathcal{B}$ is the corresponding apartement in the enlarged Bruhat-Tits building $\mathcal{B}({\bf G}, E)$.  Given an embedding $\iota\in I_\phi$, we consider the following subset of $\mathcal{B}({\bf G}, F)$:
\begin{equation}\label{eq:B:iota}
\mathcal{B}_\iota \coloneqq \bigcup_{{\bf S}\subset {\bf G}_\iota} \mathcal{A}({\bf G}, {\bf S}, F) 
\end{equation}
where $\bf S$ runs over tamely ramified maximal $F$-tori of ${\bf G}_\iota$. 

Let $I\subset I_\phi$ be a $G$-conjugacy class of $F$-rational embeddings. Let $x\in \mathcal{B}({\bf G}, F)$ and write $I_x=\left\lbrace \iota\in I : x\in\mathcal{B}_\iota\right\rbrace$. If $I_x$ is not empty, Dat assigns to each $\iota\in I_x$ an open compact subgroup $\overrightarrow{\bf G}_{\iota,x}^+$ of $G$, under certain conditions on $p$, as follows. Let $I_F^r$, for $r\in\mathbb{R}_{\geq 0}$, denote the ramification subgroups of $W_F$ in the upper numbering and put $I_F^{r+}\coloneqq \overline{\bigcup_{s>r}I_F^s}$. We also use the notation $\widetilde{\mathbb{R}}\coloneqq \mathbb{R} \sqcup \left\lbrace r+, r\in\mathbb{R}\right\rbrace$, ordered by letting $r< r+< s$ for any $r<s \in\mathbb{R}$. For each $r\in\widetilde{\mathbb{R}}_{>0}$, we set $\hat{\bf G}_{\phi,r}\coloneqq C_{\hat{\bf G}}(\phi(I_F^r))^\circ$ and $\hat{\bf S}_{\phi,r}\coloneqq \hat{\bf G}_{\phi, r, {\rm ab}}$, the cocenter of $ \hat{\bf G}_{\phi, r}$. There is a well-defined outer action $W\longrightarrow \Out(\hat{\bf G}_{\phi, r})$, defining a quasi-split reductive $F$-group ${\bf G}_{\phi, r}$ and a dual $F$-torus ${\bf S}_{\phi, r}$ with a canonical isomorphism $\hat{\bf S}_{\phi,r}\to Z({\bf G}_{\phi, r})^\circ$ (cf., \cite{dat2024}). The inclusion $\hat{\bf G}_{\phi}\subset \hat{\bf G}_{\phi, r}$ induces an $F$-rational embedding ${\bf S}_{\phi, r}\hookrightarrow {\bf S}_\phi$. We denote by $0<r_0<\cdots < r_{d-1}$ the jumps of the decreasing filtration $({\bf S}_{\phi, r})_{r>0}$ of ${\bf S}_\phi$: 
\[
\left\lbrace r_0, \ldots, r_{d-1}\right\rbrace =  \left\lbrace r>0, {\bf S}_{\phi, r+}\subsetneq {\bf S}_{\phi, r}\right\rbrace = \left\lbrace r>0, C_{\hat{\bf{G}}}(\phi(I_F^{r+}))^\circ \supsetneq C_{\hat{\bf{G}}}(\phi(I_F^r))^\circ \right\rbrace.
\] 
We also put $r_{-1}\coloneqq 0$ and $r_d\coloneqq {\rm depth} (\phi) \coloneqq \inf \left\lbrace r>0, \phi(I_F^r)=\lbrace 1\rbrace \right\rbrace$. Now fix an $F$-rational embedding $\iota \colon {\bf S}_\phi \hookrightarrow {\bf G}$ and put 
\[
{\bf G}_\iota ^i \coloneqq C_{\bf G}(\iota ({\bf S}_{\phi, r_i})) \mbox{ for } i=0,\ldots, d-1 \mbox{ and } {\bf G}_\iota^d \coloneqq {\bf G}.  
\]
Let us fix $x\in \mathcal{B}_{\iota}$. We consider the Moy-Prasad filtration $\lbrace G_{x,r}={\bf G}(F)_{x,r}, r\in\widetilde{\mathbb{R}}_{\geq 0}\rbrace$ of the stabilizer $G_x={\bf G}(F)_x$ \cite{moy1994unrefined}. For each $i=0,\ldots, d$, we put 
\[
G_{\iota, x, r}^i \coloneqq G_{x,r}\cap {\bf G}_\iota^i(F).
\]
This procedure yields the group $\overrightarrow{\bf G}_{\iota,x}^+$ defined by 
\[
\overrightarrow{\bf G}_{\iota,x}^+\coloneqq G_{\iota, x, 0+}^0 G_{\iota, x, (r_0/2)+}^1\cdots G_{\iota, x, (r_{d-1}/2)+}^d.
\]
Following Yu \cite{yu2001construction}, the parameter $\phi$ determines a character $\check{\phi}_{\iota,x}^+$ of $\overrightarrow{\bf G}_{\iota,x}^+$ (cf., \cite{dat2024}). 
Let $e_{\iota,x}^\phi$ be the idempotent in the Hecke algebra $\mathcal{H}(G)$ corresponding to the character $\check{\phi}_{\iota,x}^+$. For all $g\in G$, we have $(\overrightarrow{\bf G}_{\iota,x}^+)^g=\overrightarrow{\bf G}_{g\iota,gx}^+$ and 
\[
(e_{\iota,x}^\phi)^g = e_{g\iota, gx}^\phi.
\]
We now define the idempotent attached to $x$ by 
\begin{equation}\label{eq:e_x}
e_x=e_{I,x}^\phi\coloneqq\sum_{\iota\in I_x/\sim_x} e_{\iota,x}^\phi,
\end{equation}
where the relation $\sim_x$ is such that $\iota\sim_x\iota'$ if $\iota'\in \overrightarrow{\bf G}_{\iota,x}\cdot \iota$. If $I_x$ is empty, we take $e_x$ to be zero. The idempotent $e_x$ only depends on the image of $x$ in the reduced building $\mathcal{B}'$ of $G$. Thus, we may define $e_x$ unambiguously for any $x\in\mathcal{B}'$. Let us denote by $\mathcal{B}_\iota'$ the image of $\mathcal{B}_\iota$ in $\mathcal{B}'$. Then, by construction, we see that 
\[
e_x\not=0 \quad \mbox{if and only if}\quad x\in G\cdot \mathcal{B}_\iota'.
\]
Now, as noted in [\emph{loc.\,cit.}], the idempotents $e_x$ are not constant on the usual facets of $\mathcal{B}'$, and to this it is necessary to refine the polysimplicial structure of $\mathcal{B}'$. One way to do this is by taking sufficiently large subdivisions as in [\emph{loc,\,cit.}, 2.7.2]. Actually, according to [\emph{loc,\,cit.}, Lemma 2.7.3], if $E/F$ is a tamely ramified field extension such that ${\bf G}_\iota$ is $E$-split and such that 
\begin{equation}\label{eq:cond:r_i:E}
\lbrace r_0/2, \ldots, r_{d-1}/2\rbrace \subset v(E^\times),
\end{equation}
where $v$ is the unique valuation on $E$ extending the normalized valuation on $F$, then the idempotents $e_x$ are constant on $e$-facets, where $e$ is the ramification index of $E/F$. Thus, if we denote by $\mathcal{B}'_{\bullet/e, I,\phi}$ the polysimplicial complex structure on $G\cdot \mathcal{B}'_\iota$ induced by the intersection of the $e$-facets of $\mathcal{B}'$ with $G\cdot \mathcal{B}'_\iota$, we obtain a coefficient system $\sigma\mapsto e_\sigma\mathcal{V}$ over $\mathcal{B}'_{\bullet/e,I,\phi}$ for each smooth complex $G$-module $\mathcal{V}$.

Our aim is to compare this coefficient systems with those defined by Broussous-Schneider $\sigma\mapsto \mathcal{V}[\sigma]$ when $\bf G=\GL_N$. In this case, the working assumption of Dat is that $p\nmid N$. 
\begin{lem*}
    Let $\bf G$ be $\GL_N$ and assume that $p\nmid N$. Let $(J,\lambda)$ be a simple type of $G$. Let $\phi$ be the image in $\Phi(P_F,G)$ of the Langlands parameter $\varphi\in H^1(W_F, \hat{\bf G})$ of any irreducible representation in $\mathcal{R}_{(J,\lambda)}(G)$. Let $I$ be the $G$-conjugacy class of an $F$-rational embedding $\iota\colon {\bf S}_\phi \hookrightarrow {\bf G}$ such that ${\bf G}_\iota$ is naturally isomorphic to ${\bf L}_\phi$. Then: 
    \begin{itemize}
        \item[(i)] As polysimplical complexes, we have 
        \[
        X[E] = \mathcal{B}'_{\bullet/e, I, \phi},
        \]
        where $E=F[\beta]$ and $e$ is the ramification index of $E/F$;
        \item[(ii)] For each simplex $\sigma$, we have 
        \[
        \mathcal{V}[\sigma] = e_x \mathcal{V}, 
        \]
        for all $x$ in the interior of $\sigma$ and for every $\mathcal{V}\in\mathcal{R}_{(J,\lambda)}(G)$. In particular, the spaces $e_x\mathcal{V}$ are constant on $e$-facets.
    \end{itemize}
\end{lem*}

\begin{proof}
First, as shown in \cite{bushnell2003localIV}, the wild inertial parameter $\phi$ is well-defined and depends only on the endo-class $\Theta_0$ contained in $(J,\lambda)$. With this $\phi$, we have \[
{\bf S}_\phi=\Res_{E/F}\mathbb{G}_m,\quad {\bf L}_{\phi}=\Res_{E/F}\GL_{N/[E:F]},
\]
and $I$ is the conjugacy class of the embedding $\iota\colon \Res_{E/F}\mathbb{G}_m\hookrightarrow \GL_N$ yielding the fixed embedding $E^\times\hookrightarrow G$. As above, we denote by $X$ the reduced building of $G=\GL_N(F)$. Since ${\bf S}_\phi$ is $F$-anisotropic modulo the center $\mathbb{G}_m$ of $\GL_N$, the image of the reduced builing of ${\bf L}_\phi(F)$ in $X$ coincides with the projection to $X$ of the image of the enlarged building of ${\bf L}_\phi(F)$ in that of $\bf G$. Thus, as sets, we obtain $X[E]=G\cdot\mathcal{B}'_\iota=\mathcal{B}'_{\bullet/e, I, \phi}$. As for the simplicial structure, it is known that the simplicial structure of $X_E$ coincides with the intersection of the $e^{\rm th}$-subdivision of $X$ with $X_E$ \cite[Theorem 1.1 (A)]{broussous2002building}. Hence, since the action of $G$ is simplicial, $X[E] = \mathcal{B}'_{\bullet/e, I, \phi}$ as polysimplicial complexes.

It remains to verify (ii). To do so, by $G$-equivariance, we may take $x\in X_E$. Note that by \cite{mayeux2020comparison} we have that $\overrightarrow{\bf G}_{\iota,x}^+=H^1(\mathfrak{B})$ and $\check{\phi}_{\iota,x}^+=\Theta_0(\mathfrak{B})$, where $\mathfrak{B}$ denotes the hereditary $\mathfrak{o}_E$-order corresponding to $\sigma$. (Since $p\nmid N$, the extension $E/F$ is tamely ramified.) Moreover, for all $g\in G$, we have  $g\iota\in I_x$ if and only if $g\in U(\mathfrak{A}(\mathfrak{B}))$. Therefore, $e_{g\iota, x}^\phi=e_{g\iota, gx}^\phi=(e_{\iota, x}^\phi)^g$ for $g\iota\in I_x$, leading to
\[
e_x\mathcal{V}= \sum_{g\in U(\mathfrak{A}(\mathfrak{B}))}g\mathcal{V}^{\check{\phi}_{\iota,x}^+}=\sum_{g\in U(\mathfrak{A}(\mathfrak{B}))}g\mathcal{V}^{\Theta_0(\mathfrak{B})}.
\]
Then, assertion (ii) is equivalent to 
 \begin{equation}\label{eq:coeff:dat}
        \mathcal{V}[\sigma]=\sum_{g\in U(\mathfrak{A}(\mathfrak{B}))}g\mathcal{V}^{\Theta_0(\mathfrak{B})}.
    \end{equation}
By construction, 
\[
\mathcal{V}[\sigma] = \sum_{g\in U(\mathfrak{A}(\mathfrak{B}))/U^1(\mathfrak{A}(\mathfrak{B}))}g\mathcal{V}^{\eta(\mathfrak{A}(\mathfrak{B}))},
\]
where $\eta(\mathfrak{A}(\mathfrak{B}))=\Ind_{J^1(\mathfrak{B})}^{U^1(\mathfrak{A}(\mathfrak{B}))}\eta(\mathfrak{B})$. We observe that the $G$-intertwining of $\eta(\mathfrak{B})$ is given by $J^1(\mathfrak{B})G_EJ^1(\mathfrak{B})$ (\cite[(5.1.1)]{bushnell1993admissible}), and $
J^1(\mathfrak{B})G_EJ^1(\mathfrak{B})\cap U^1(\mathfrak{A}(\mathfrak{B}))= J^1(\mathfrak{B})$, so the representation $\eta(\mathfrak{A}(\mathfrak{B}))$ is irreducible. Hence, we obtain 
\[
\mathcal{V}^{\eta(\mathfrak{A}(\mathfrak{B}))}= U^1(\mathfrak{A}(\mathfrak{B}))\cdot\mathcal{V}^{\eta(\mathfrak{B})}.
\]
Now, since $\eta(\mathfrak{B})$ is $\Theta_0(\mathfrak{B})$-isotypic and $\Ind_{H^1(\mathfrak{B})}^{J^1(\mathfrak{B})}\Theta_0(\mathfrak{B})$ is $\eta(\mathfrak{B})$-isotypic, it follows that $\mathcal{V}^{\eta(\mathfrak{B})}=\mathcal{V}^{\Theta_0(\mathfrak{B})}$. Therefore, \eqref{eq:coeff:dat} holds, and the lemma follows.
\end{proof}

\begin{rem}
    Although the extension  $E/F$ considered above is not necessarily Galois, if  $E'/F$ denotes its Galois closure, then  $E'/E$ is unramified. In particular, the ramification indices agree: $e=e(E/F) = e(E'/F)$. The tamely ramified Galois extension $E'/F$ clearly splits the group  $\mathbf{G}_\iota = \Res_{E/F} \GL_{N/[E:F]}$; however, it does not necessarily satisfy condition~\eqref{eq:cond:r_i:E}. In fact, according to \cite{mayeux2020comparison}, if $\lbrace [\mathfrak{A}_0, 0, \tilde{r}_i, \beta_i] \mid 0 \leq i \leq s \rbrace$ is a defining sequence of the underlying simple stratum attached to  $(J,\lambda)$, then $s = d$ and  $r_i = -v(\beta_i - \beta_{i+1})$ for $i = 0, \ldots, d-1$. (In \emph{loc.\,cit.}, the valuation $v$ on $E$ extending the normalized valuation on $F$ is denoted by $\mathrm{ord}$.) It follows that  $\{ r_0, \ldots, r_{d-1} \} \subseteq v(E^\times) = v(E'^\times)$, but the set $\{ r_0/2, \ldots, r_{d-1}/2 \}$ may not lie in $v(E^\times)$. For this reason, a quadratic ramified extension of $E'$ may be required. Consequently, the idempotent $e_x$ may not be constant on $e$-facets, but rather on $2e$-facets. Nevertheless, as we have shown, the spaces $e_x \mathcal{V}$ are constant on $e$-facets for all  $\mathcal{V} \in \mathcal{R}_{(J,\lambda)}(G)$.
\end{rem}

\section{Acyclicity}\label{sec:main}
Having reviewed the definition of the coefficient system attached to $\mathcal{V}$, we now consider the chain complex derived from it and prove that it gives an exact resolution of $\mathcal{V}$ in $\mathcal{R}_{(J,\lambda)}(G)$.

\subsection{} We recall that the coefficient system $\mathcal{C}=\mathcal{C}[\Theta_0, V, \mathcal{V}]$ is defined on the simplicial complex $X[L]$, which has dimension
\[
d\coloneqq \dim X[L] = \dim_F V/[L:F] -1 = e_0-1. 
\]
For each non-negative integer $q\leq d$, we denote by $X[L]_q$ the set of $q$-simplices of $X[L]$. An ordered $q$-simplex in $X[L]$ is a sequence $(\sigma_0,\ldots, \sigma_q)$ of vertices such that $\left\lbrace \sigma_0,\ldots, \sigma_q\right\rbrace$ is a $q$-simplex. Two such ordered simplices are called equivalent if they differ by an even permutation of the vertices. The corresponding equivalence classes are called oriented $q$-simplices and are denoted by $\langle \sigma_0,\ldots, \sigma_q\rangle$. We let $X[L]_{(q)}$ be the set of oriented $q$-simplices of $X[L]$. 

The space $C_q^{\rm or}(X[L],\mathcal{C})$ of oriented $q$-chains of $X[L]$ with coefficients in $\mathcal{C}$ is the $\mathbb{C}$-vector space of all maps
\[
\omega\colon X[L]_{(q)} \longrightarrow \mathcal{V}
\]
such that 
\begin{itemize}
    \item[-]$\omega$ has finite support, 
    \item[-] $\omega(\langle \sigma_0,\ldots, \sigma_q\rangle)\in \mathcal{V}[\left\lbrace \sigma_0,\ldots, \sigma_q\right\rbrace]$, and 
    \item[-] $\omega(\langle \sigma_{\iota(0)},\ldots, \sigma_{\iota(q)}\rangle)={\rm sgn}(\iota)\cdot\omega(\langle \sigma_0,\ldots, \sigma_q\rangle)$ for any permutation $\iota$. 
\end{itemize}
The group $G$ acts smoothly on $C_q^{\rm or}(X[L],\mathcal{C})$ via 
\[
(g\omega)(\langle \sigma_0,\ldots, \sigma_q\rangle)\coloneqq g(\omega(\langle g^{-1}\sigma_0,\ldots, g^{-1}\sigma_q\rangle)).
\]
With respect to the $G$-equivariant boundary maps 
\begin{equation*}
\begin{split}
    \partial\colon C_{q+1}^{\rm or}(X[L],\mathcal{C}) & \longrightarrow   C_{q}^{\rm or}(X[L],\mathcal{C}) \\
    \omega &\longmapsto [\langle \sigma_0,\ldots, \sigma_q\rangle \mapsto \sum_{\left\lbrace \sigma, \sigma_0,\ldots, \sigma_q\right\rbrace\in X[L]_{q+1}}\omega(\langle \sigma,\sigma_0,\ldots, \sigma_q\rangle)]
\end{split}
\end{equation*}
we then have the augmented chain complex in $\mathcal{R}(G)$
\begin{equation}\label{eq:chain:complex}
C_d^{\rm or}(X[L],\mathcal{C})\xrightarrow{\ \partial\ }\cdots \xrightarrow{\ \partial\ }C_0^{\rm or}(X[L],\mathcal{C})\xrightarrow{\ \epsilon\ }\mathcal{V}
\end{equation}
where $\epsilon(\omega)=\sum_{\sigma\in X[L]_{(0)}}\omega(\sigma)\in\mathcal{V}$. 

\subsection{} By \cite[Proposition IX.2]{BS2017type}, the chain complex \eqref{eq:chain:complex} is contained within the Bernstein component $\mathcal{R}_{(J,\lambda)}(G)$. Theorem \ref{ss:main} below will establish that this complex is exact. To demonstrate this, we employ the strategy outlined in \S\,X of [\textit{loc.\,cit.}].

Let us begin by fixing a maximal hereditary $\mathfrak{o}_E$-order $\mathfrak{B}_{\rm max}$ in $B$ that contains $\mathfrak{B}_0$. Given that $L/E$ is unramified, the order $\mathfrak{C}_{\rm max}\coloneqq C\cap\mathfrak{B}_{\rm max}$ is maximal in $C={\rm End}_LV$. Following \cite[\S\,5]{bushnell1993admissible}, we consider the family of compatible $\beta$-extensions $\left\lbrace \kappa(\mathfrak{B})\right\rbrace_{\mathfrak{B}_0\subseteq\mathfrak{B}\subseteq\mathfrak{B}_{\rm max}}$, where $\mathfrak{B}$ runs over orders in $B$ between $\mathfrak{B}_0$ and $\mathfrak{B}_{\rm max}$. These representations are characterized by
\[
\Ind_{J(\mathfrak{B})}^{U(\mathfrak{B})U^1(\mathfrak{A}(\mathfrak{B}))}\kappa(\mathfrak{B})\cong \Ind_{U(\mathfrak{B})J^1(\mathfrak{B}_{\rm max})}^{U(\mathfrak{B})U^1(\mathfrak{A}(\mathfrak{B}))}\kappa(\mathfrak{B}_{\rm max}) \quad \mbox{and}\quad \kappa(\mathfrak{B}_0)=\kappa.
\]
Fix a non-negative integer $q\leq d$ and consider the set $X[L]_q$ of unoriented $q$-simplices in $X[L]$. The space $C_q(X[L]_q,\mathcal{C})$ of (unoriented) $q$-chains of $X[L]$ with coefficients in $\mathcal{C}$ is the $\mathbb{C}$-vector space of all maps $\omega\colon X[L]_q\to \mathcal{V}$ such that $\omega$ has finite support and $\omega(\sigma)\in\mathcal{V}[\sigma]$ for all $\sigma\in X[L]_q$. The group $G$ acts smoothly on $C_q(X[L]_q,\mathcal{C})$ via $(g\omega)(\sigma)=g(\omega(g^{-1}\sigma))$. 

Let $J_{\rm max}\coloneqq U(\mathfrak{B}_0)J^1(\mathfrak{B}_{\rm max})$ and let us denote by $\lambda_{\rm max}$ the irreducible representation of $J_{\rm max}$ given by
\begin{equation}
\lambda_{\rm max} \coloneqq \kappa(\mathfrak{B}_{\rm max})\otimes \rho.
\end{equation}
We denote by $\Omega_q$ the set of orbits of $J_{\rm{max}}$ in $X[L]_q$. For each orbit $\Sigma\in\Omega_q$, we define $C_q(\Sigma, \mathcal{C})\subset C_q(X[L],\mathcal{C})$ as the subspace of $q$-chains supported in $\Sigma$. As a $J_{\rm{max}}$-module, $C_q(X[L],\mathcal{C})$ decomposes as 
\begin{equation*}
    C_q(X[L],\mathcal{C})= \coprod_{\Sigma\in \Omega_q} C_q(\Sigma,\mathcal{C}).
\end{equation*}
For each $\Sigma\in \Omega_q$, there exists a hereditary $\mathfrak{o}_L$-order $\mathfrak{C}$ in $C$ with $\mathfrak{C}_{\rm min}\subseteq \mathfrak{C}\subseteq \mathfrak{C}_{\rm max}$ and an element $x\in G$ such that $\Sigma=J_{\rm{max}}x.\sigma_{\mathfrak{C}}$, where $\sigma_{\mathfrak{C}}$ is the simplex in $X[L]_q$ corresponding to $\mathfrak{C}$. We have the disjoint union 
\begin{equation*}
    \Sigma=\bigcup_{j\in J_{\rm{max}}/J_{\rm{max}}\cap U(\mathfrak{A})^x}\conj{jx\sigma_{\mathfrak{C}}},
\end{equation*}
where $\mathfrak{A}=\mathfrak{A}(\mathfrak{C})$ is the order in $A$ such that $\mathfrak{C}=C\cap \mathfrak{A}$, and the union is taken over a fixed set of representatives of $J_{\rm{max}}/J_{\rm{max}}\cap U(\mathfrak{A})^x$. Then we have an isomorphism of $J_{\rm{max}}$-modules 
\begin{equation*}
    C_q(\Sigma,\mathcal{C})=\coprod_{j\in J_{\rm{max}}/J_{\rm{max}}\cap U(\mathfrak{A})^x}jxC_q(\sigma_{\mathfrak{C}},\mathcal{C}).
\end{equation*}
The action of $U(\mathfrak{A})$ on $C_q(\sigma_{\mathfrak{C}},\mathcal{C})$ gives an action of $U(\mathfrak{A})^x$ on $xC_q(\sigma_{\mathfrak{C}},\mathcal{C})$, and we have an isomorphism of $J_{\rm max}$-modules
\begin{equation}\label{eq:dec:Cq:Sigma}
     C_q(\Sigma,\mathcal{C})= \Ind_{J_{\rm max}\cap U(\mathfrak{A})^x}^{J_{\rm max}} xC_q(\sigma_{\mathfrak{C}},\mathcal{C}).
\end{equation}
Actually, for each representative $j\in J_{\rm{max}}/J_{\rm{max}}\cap U(\mathfrak{A})^x$, the space $jxC_q(\sigma_{\mathfrak{C}},\mathcal{C})$ can be identified with the subspace of $\Ind_{J_{\rm max}\cap U(\mathfrak{A})^x}^{J_{\rm max}} xC_q(\sigma_{\mathfrak{C}},\mathcal{C})$ consisting of functions supported on the coset $jJ_{\rm{max}}\cap U(\mathfrak{A})^x$. Now, by construction, we have 
\[
C_q(\sigma_{\mathfrak{C}},\mathcal{C})= \mathcal{V}[\sigma_{\mathfrak{C}}]
\]
as $U(\mathfrak{A})$-modules. By \cite[V.5, V.10]{BS2017type}, 
\begin{equation}\label{eq:V:sigma_Sigma}
    \mathcal{V}[\sigma_{\mathfrak{C}}]=\sum_{g\in U(\mathfrak{A})/U(\mathfrak{B})J^1(\mathfrak{B}_{{\rm max}})} g\mathcal{V}^{\eta(\mathfrak{B},\mathfrak{B}_{\rm max})},
\end{equation}
where $\mathfrak{B}=B\cap\mathfrak{A}$ and $\eta(\mathfrak{B},\mathfrak{B}_{\rm max})$ is the representation of $U^1(\mathfrak{B})J^1(\mathfrak{B}_{\rm max})$ described in \cite[\S\,III.1]{BS2017type}.

\subsection{} Consider the space $\mathcal{V}^{\eta(\mathfrak{B},\mathfrak{B}_{\rm max})}$ as a $U(\mathfrak{B})J^1(\mathfrak{B}_{\rm max})$-module. (See the comments at the beginning of the proof of \cite[Lemma IX.5]{BS2017type} for the $U(\mathfrak{B})J^1(\mathfrak{B}_{\rm max})$-invariance of $\mathcal{V}^{\eta(\mathfrak{B},\mathfrak{B}_{\rm max})}$.) By \S\,XI [\textit{loc.\,cit.}], there exist irreducible representations $\tau_1,\ldots, \tau_m$ (not necessarily pairwise disjoint) which are the inflation of irreducible representations of $U(\mathfrak{B})/U^1(\mathfrak{B})$, such that
\begin{equation}\label{eq:def:tau_i}
    \mathcal{V}^{\eta(\mathfrak{B},\mathfrak{B}_{\rm max})}=\bigoplus_{i=1}^m\kappa(\mathfrak{B}_{\rm max})\otimes \tau_i.
\end{equation}
For each $i$ we have 
\begin{equation}\label{eq:prop:tau_i}
        \Hom_{U(\mathfrak{B})J^1(\mathfrak{B}_{\rm max})}(\kappa(\mathfrak{B}_{\rm max})\otimes \tau_i,\mathcal{V})= \Hom_{\mathbb{P}_{\mathfrak{B}}}(\tau_i, \mathcal{V}(\mathfrak{B}_{\rm max})_{\mathbb{U}_{\mathfrak{B}}}),
\end{equation}
where $\mathcal{V}(\mathfrak{B}_{\rm max})$ is the $U(\mathfrak{B}_{\rm max})/U^1(\mathfrak{B}_{\rm max})$-module \[
\Hom_{J^1(\mathfrak{B}_{\rm max})}(\kappa(\mathfrak{B}_{\rm max}),\mathcal{V}), 
\]
and $\mathbb{U}_{\mathfrak{B}}$ is the unipotent radical of the parabolic subgroup $\mathbb{P}_{\mathfrak{B}}(k_E)$ of $\mathbb{G}(k_E)$ given by $U(\mathfrak{B})J^1(\mathfrak{B}_{\rm max})/J^1(\mathfrak{B}_{\rm max})$ (see \ref{ss:her:ord}). Then, by \cite[\S\,5 Proposition 3]{schneider-zink1999}, the representation $\tau_i$ has cuspidal support $\rho$, seeing $U(\mathfrak{B}_0)/U^1(\mathfrak{B}_0)$ as a Levi of $U(\mathfrak{B})/U^1(\mathfrak{B})$. 

We now prove some results involving the representations $\tau_1,\ldots,\tau_m$ which will be needed later. If $\mathcal{V}$ lies in the discrete series, then, as noted in \cite[Lemma XI.1.3]{BS2017type}, we have $\mathcal{V}^{\eta(\mathfrak{B},\mathfrak{B}_{\rm max})}=\kappa(\mathfrak{B}_{\rm max})\otimes {\rm St}(\mathfrak{B}_{\rm max},\rho)$, and the corresponding results are proved (see Proposition XI.1.4, Lemma XI.2.2, and Lemma XI.2.3 [\emph{loc.\,cit.}]). Here, for completeness, we verify that their arguments apply in our more general setting. We begin by ordering the set $\lbrace \tau_i\rbrace_{1\leq i\leq m}$ so that there exists some $m'\leq m$ such that \[
\kappa(\mathfrak{B}_{\rm max})\otimes \tau_i\not\simeq \kappa(\mathfrak{B}_{\rm max})\otimes \tau_j \mbox{ for } 1\leq i<j\leq m',
\]
and, for all $m'<i\leq m$, there exists $i'\leq m'$ such that $\kappa(\mathfrak{B}_{\rm max})\otimes \tau_i\simeq \kappa(\mathfrak{B}_{\rm max})\otimes \tau_{i'}$. 

\subsubsection{}\label{prop:XI.1.4}\begin{Proposition}
    \begin{itemize}
        \item[\emph{(i)}] \emph{The $U(\mathfrak{A})$-intertwining of $\kappa(\mathfrak{B}_{\rm max})\otimes \tau_i$ is equal to $U(\mathfrak{B})J^1(\mathfrak{B}_{\rm max})$, for each $i=1,\ldots,m$.} 
            \item[\emph{(ii)}]\emph{ For each $i$, the representation of $U(\mathfrak{A})$ given by} 
        \begin{equation*}
            \lambda_i(\sigma_{\mathfrak{C}})\coloneqq\Ind_{U(\mathfrak{B})J^1(\mathfrak{B}_{\rm max})}^{U(\mathfrak{A})}\kappa(\mathfrak{B}_{\rm max})\otimes \tau_i
        \end{equation*}
        \emph{is irreducible.} 
        \item[\emph{(iii)}]\emph{ We have} 
        \begin{equation*}
            \mathcal{V}[\sigma_{\mathfrak{C}}]=\bigoplus_{i=1}^{m'}\mathcal{V}^{\lambda_i(\sigma_{\mathfrak{C}})}\simeq \bigoplus_{i=1}^m\lambda_i(\sigma_{\mathfrak{C}}),
        \end{equation*}
        \emph{where the isomorphism is an isomorphism of $U(\mathfrak{A})$-modules.}
    \end{itemize} 
\end{Proposition}

\begin{proof}
For each $i$, the restriction of $\kappa(\mathfrak{B}_{\rm max})\otimes \tau_i$ to $U^1(\mathfrak{B})J^1(\mathfrak{B}_{\rm max})$ is a multiple of $\eta(\mathfrak{B},\mathfrak{B}_{\rm max})$, and points (i) and (ii) follow exactly as in \cite[Proposition XI.1.4 (i)-(ii)]{BS2017type}. Let us verify (iii). Since $\lambda_i(\sigma_{\mathfrak{C}})=\Ind_{U(\mathfrak{B})J^1(\mathfrak{B}_{\rm max})}^{U(\mathfrak{A})}\kappa(\mathfrak{B}_{\rm max})\otimes \tau_i$ is irreducible, by Lemma V.2 [\textit{loc.\,cit.}] we have 
\begin{equation*}
\mathcal{V}^{\lambda_i(\sigma_{\mathfrak{C}})}=\sum_{g\in U(\mathfrak{A})}g\mathcal{V}^{\kappa(\mathfrak{B}_{\rm max})\otimes \tau_i}.
\end{equation*}
By the arrangement of $(\kappa(\mathfrak{B}_{\rm max})\otimes \tau_i)$, there exist positive integers $m_i$, $i=1,\ldots,m'$, such that 
\begin{equation}\label{eq:eta:isotypic:dec}
    \mathcal{V}^{\eta(\mathfrak{B},\mathfrak{B}_{\rm max})}=\bigoplus_{i=1}^m\kappa(\mathfrak{B}_{\rm max})\otimes \tau_i=\bigoplus_{i=1}^{m'}m_i(\kappa(\mathfrak{B}_{\rm max})\otimes \tau_i),
\end{equation}
where $m_i (\kappa(\mathfrak{B}_{\rm max})\otimes \tau_i)=\bigoplus_{j=1}^{m_i}\kappa(\mathfrak{B}_{\rm max})\otimes \tau_i$. Now, since each $\kappa(\mathfrak{B}_{\rm max})\otimes \tau_i$ is $\eta(\mathfrak{B},\mathfrak{B}_{\rm max})$-isotypic, we have $\mathcal{V}^{\kappa(\mathfrak{B}_{\rm max})\otimes \tau_i}\subseteq \mathcal{V}^{\eta(\mathfrak{B},\mathfrak{B}_{\rm max})}$. By \eqref{eq:eta:isotypic:dec}, we have \[
\mathcal{V}^{\kappa(\mathfrak{B}_{\rm max})\otimes \tau_i}=m_i (\kappa(\mathfrak{B}_{\rm max})\otimes \tau_i), \quad   i=1,\ldots, m'. 
\]
Thus, we have 
\begin{equation*}
\mathcal{V}^{\lambda_i(\sigma_{\mathfrak{C}})}=\sum_{g\in U(\mathfrak{A})}g\mathcal{V}^{\kappa(\mathfrak{B}_{\rm max})\otimes \tau_i}=\sum_{g\in U(\mathfrak{A})}g(m_i(\kappa(\mathfrak{B}_{\rm max})\otimes \tau_i)), \quad i=1,\ldots,m'.
\end{equation*}
The equality in (iii) follows then from \eqref{eq:V:sigma_Sigma} and \eqref{eq:eta:isotypic:dec}. For the isomorphism, it is enough to show that $\mathcal{V}^{\lambda_i(\sigma_{\mathfrak{C}})}=m_i(\lambda_i(\sigma_{\mathfrak{C}})$ for $i=1,\ldots,m'$, i.e., that
\begin{equation*}
    \dim \Hom_{U(\mathfrak{A})}(\lambda_i(\sigma_{\mathfrak{C}}),\mathcal{V})=m_i.
\end{equation*}
The latter is a direct consequence of Frobenius reciprocity. 
\end{proof}

\subsubsection{}\label{ss:lemma:hom:1}\begin{Lemma}
    If $x\in G_L$, then 
    \begin{equation}\label{eq:lemma:hom:1}
        \Hom_{J_{\rm max}}(\lambda_{\rm max}, C_q(\Sigma, \mathcal{C}))=\bigoplus_{i=1}^m\Hom_{J_{\rm max}\cap U(\mathfrak{B})^xJ^1(\mathfrak{B}_{\rm max})^x}(\lambda_{\rm max},\kappa(\mathfrak{B}_{\rm max})^x\otimes\tau_i^x). 
    \end{equation}
\end{Lemma}

\begin{proof}
   Recall that 
\[
C_q(\Sigma,\mathcal{C})= \Ind_{J_{\rm max}\cap U(\mathfrak{A})^x}^{J_{\rm max}} xC_q(\sigma_{\mathfrak{C}},\mathcal{C}).
\]
(See \eqref{eq:dec:Cq:Sigma} above.) Now, by Proposition \ref{prop:XI.1.4}, 
\begin{equation*}
C_q(\sigma_{\mathfrak{C}},\mathcal{C})=x\bigoplus_{i=1}^m \Ind_{U(\mathfrak{B})J^1(\mathfrak{B}_{\rm max})}^{U(\mathfrak{A})}(\kappa(\mathfrak{B}_{\rm max})\otimes \tau_i)
\end{equation*}
as $U(\mathfrak{A})$-modules. Thus, we have
\begin{equation*}
   xC_q(\sigma_{\mathfrak{C}},\mathcal{C}) =\bigoplus_{i=1}^m \Ind_{U(\mathfrak{B})^xJ^1(\mathfrak{B}_{\rm max})^x}^{U(\mathfrak{A})^x}(\kappa(\mathfrak{B}_{\rm max})^x\otimes \tau_i^x).
\end{equation*}
Mackey's restriction formula gives
\begin{equation*}
xC_q(\sigma_{\mathfrak{C}},\mathcal{C})\vert_{J_{\rm max}\cap U(\mathfrak{A})^x}
     =\bigoplus_{i=1}^m\bigoplus_{u\in U}\Ind_{J_{\rm max}\cap U(\mathfrak{B})^{ux}J^1(\mathfrak{B}_{\rm max})^{ux}}^{J_{\rm max}\cap U(\mathfrak{A})^x}(\kappa(\mathfrak{B}_{\rm max})^{ux}\otimes \tau_i^{ux}), 
\end{equation*}
where $U$ is (a set of representatives of) the double coset set
\begin{equation*}
    U=J_{\rm max}\cap U(\mathfrak{A})^x\backslash U(\mathfrak{A})^x/U(\mathfrak{A})^xJ^1(\mathfrak{B}_{\rm max})^x.
\end{equation*}
Consequently, by Frobenius reciprocity, we have
\begin{equation*}
    \Hom_{J_{\rm max}}( \lambda_{\rm max}, C_q(\Sigma,\mathcal{C})) = \bigoplus_{i=1}^m\bigoplus_{u\in U}\Hom_{J_{\rm max}\cap U(\mathfrak{B})^{ux}J^1(\mathfrak{B}_{\rm max})^{ux}}\left(\lambda_{\rm max}, \kappa(\mathfrak{B}_{\rm max})^{ux}\otimes \tau_i^{ux}\right).
\end{equation*}
 Now, let us assume that 
 \[
 \Hom_{J_{\rm max}\cap U(\mathfrak{B})^{ux}J^1(\mathfrak{B}_{\rm max})^{ux}}\left(\lambda_{\rm max}, \kappa(\mathfrak{B}_{\rm max})^{ux}\otimes \tau_i^{ux}\right)\not= 0
 \]
 for some $i$ and some $u\in U$. In order to prove \eqref{eq:lemma:hom:1}, it suffices to verify that $u$ lies within the double class $\overline{1}\in U$ corresponding to the unit element. Since $\tau_i^{ux}$ has cuspidal support $\rho^{ux}$, now seeing $U(\mathfrak{B}_0)^{ux}/U^1(\mathfrak{B}_0)^{ux}$ as a Levi of $U(\mathfrak{B})^{ux}/U^1(\mathfrak{B})^{ux}$, the representation $\kappa(\mathfrak{B}_{\rm max})^{ux}\otimes \tau_i^{ux}$ embeds in 
 \begin{equation*}
     \Ind_{U(\mathfrak{B}_0)^{ux}J^1(\mathfrak{B}_{\rm max})^{ux}}^{U(\mathfrak{B})^{ux}J^1(\mathfrak{B}_{\rm max})^{ux}}\kappa(\mathfrak{B}_{\rm max})^{ux}\otimes \rho^{ux}=\Ind_{J_{\rm max}^{ux}}^{U(\mathfrak{B})^{ux}J^1(\mathfrak{B}_{\rm max})^{ux}}\lambda_{\rm max}^{ux}
 \end{equation*}
 as a $U(\mathfrak{B})^{ux}J^1(\mathfrak{B}_{\rm max})^{ux}$-module. It follows that 
 \begin{equation}\label{eq:hom:lambda:not:0}
\Hom_{J_{\rm max}\cap U(\mathfrak{B})^{ux}J^1(\mathfrak{B}_{\rm max})^{ux}}(\lambda_{\rm max}, \Ind_{J_{\rm max}^{ux}}^{U(\mathfrak{B})^{ux}J^1(\mathfrak{B}_{\rm max})^{ux}}\lambda_{\rm max}^{ux})\not= 0.
 \end{equation}
By Mackey's restriction formula, we have 
\begin{equation*}
    \begin{split}
        \left.\left(\Ind_{J_{\rm max}^ux}^{U(\mathfrak{B})^{ux}J^1(\mathfrak{B}_{\rm max})^{ux}}  \lambda_{\rm max}^{ux} \right) \right\vert_{J_{\rm max}\cap U(\mathfrak{B})^{ux}J^1(\mathfrak{B}_{\rm max})^{ux}}
         & =\bigoplus_{v\in V_u}\Ind_{J_{\rm max}\cap U(\mathfrak{B})^{ux}J^1(\mathfrak{B}_{\rm max})^{ux}\cap J_{\rm max}^{vux}}^{J_{\rm max}\cap U(\mathfrak{B})^{ux}J^1(\mathfrak{B}_{\rm max})^{ux}}\lambda_{\rm max}^{vux} \\
        & = \bigoplus_{v\in V_u} \Ind_{J_{\rm max}\cap J_{\rm max}^{vux}}^{J_{\rm max}\cap U(\mathfrak{B})^{ux}J^1(\mathfrak{B}_{\rm max})^{ux}}\lambda_{\rm max}^{vux},
    \end{split}
\end{equation*}
where 
\[
V_u= J_{\rm max}\cap U(\mathfrak{B})^{ux}J^1(\mathfrak{B}_{\rm max})^{ux}\backslash  U(\mathfrak{B})^{ux}J^1(\mathfrak{B}_{\rm max})^{ux} / J_{\rm max}^{ux}.
\]
Then, by \eqref{eq:hom:lambda:not:0} and Frobenius reciprocity, we obtain 
\[
\bigoplus_{v\in V_u}\Hom_{J_{\rm max}\cap J_{\rm max}^{vux}}(\lambda_{\rm max}, \lambda_{\rm max}^{vux})\not = 0.
\]
Therefore, there exists $v\in U(\mathfrak{B})^{ux}J^1(\mathfrak{B}_{\rm max})^{ux} $ such that $vux$ intertwines $\lambda_{\rm max}$, i.e., $vux\in J_{\rm max}G_L J_{\rm max}$. Note that in this case we have 
\[
vux\in uxU(\mathfrak{B})J^1(\mathfrak{B}_{\rm max}) \cap J_{\rm max} G_L J_{\rm max}.
\]
In particular, the non-emptiness of
\[
uxU(\mathfrak{B})J^1(\mathfrak{B}_{\rm max}) \cap J_{\rm max} G_L J_{\rm max}
\]
implies that 
\[
u\in J_{\rm max} G_L U(\mathfrak{B})J^1(\mathfrak{B}_{\rm max}) x^{-1}.
\]
Since $x\in G_L$, then
\[
J_{\rm max} G_L U(\mathfrak{B})J^1(\mathfrak{B}_{\rm max}) x^{-1}=J_{\rm max} G_L U(\mathfrak{B})^xJ^1(\mathfrak{B}_{\rm max})^x.
\]
Hence, without changing the double class $\overline{u}$ of $u$ in $U$, we may assume that $u\in J_{\rm max}G_L$. Let us write $u=jg_L$, with $j\in J_{\rm max}$ and $g_L\in G_L$. Since $u\in U(\mathfrak{A})^x$, then 
\[
u(x\sigma_{\mathfrak{C}})=x\sigma_{\mathfrak{C}}= j(g_Lx\sigma_{\mathfrak{C}}).
\]
So $x\sigma_{\mathfrak{C}}$ and $g_Lx\sigma_{\mathfrak{C}}$ are simplices of $X_L$ conjugated under the action of $J_{\rm max}\subset U(\mathfrak{A}_{\rm min})$. By \cite[Lemma X.4.4]{BS2017type}, $x\sigma_{\mathfrak{C}}$ and $g_Lx\sigma_{\mathfrak{C}}$ are also conjugated under $U(\mathfrak{C}_{\rm min})$. Let $i\in U(\mathfrak{C}_{\rm min})$ such that $x\sigma_{\mathfrak{C}}=ig_Lx\sigma_{\mathfrak{C}}$. Then $ig_L\in U(\mathfrak{A})^x\cap G_L=U(\mathfrak{C})^x$, which implies that $g_L\in U(\mathfrak{C}_{\rm min})U(\mathfrak{C})^x$. It follows that $u\in J_{\rm max}U(\mathfrak{C}_{\rm min})U(\mathfrak{C})^x=J_{\rm max}U(\mathfrak{C})^x$, but then \[
u\in J_{\rm max}U(\mathfrak{C})^x \cap U(\mathfrak{A})^x= \left(J_{\rm max}\cap U(\mathfrak{A})^x\right)\cdot U(\mathfrak{C})^x,
\]
and the double class $\overline{u}$ of $u$ is $\overline{1}$. Lemma \ref{ss:lemma:hom:1} follows. 
\end{proof}

\subsubsection{}\label{ss:lemma:hom:2}\begin{Lemma}
    For all $x\in G_L$ and for all $1\leq i\leq m$, we have 
    \[
    \dim \Hom_{J_{\rm max}\cap U(\mathfrak{B})^xJ^1(\mathfrak{B}_{\rm max})^x}(\lambda_{\rm max}, \kappa(\mathfrak{B}_{\rm max})^x\otimes\tau_i^x)=\dim\Hom_{U(\mathfrak{B}_0)\cap U(\mathfrak{B})^x}(\rho, \tau_i^x).
    \]
\end{Lemma}

\begin{proof}
   It follows exactly as in \cite[Lemma XI.2.3]{BS2017type}. Denote by $Y$ (resp. $X_0$, $X$) for the space of $\kappa(\mathfrak{B}_{\rm max})$ (resp. $\rho$, $\tau_i$). Let $\varphi\in\Hom_{\mathbb{C}}(Y\otimes X_0, Y\otimes X)={\rm End}_{\mathbb{C}}(Y)\otimes \Hom_{\mathbb{C}}(X_0,X)$. Write
   \[
   \varphi = \sum_{i\in I} S_i\otimes T_i,
   \]
   for some finite index set $I$, where $S_i\in{\rm End}_{\mathbb{C}}(Y)$, $T_i\in\Hom_{\mathbb{C}}(X_0,X)$, and where the $T_i$ are linearly independent. Then $\varphi$ intertwines $\kappa(\mathfrak{B}_{\rm max})\otimes \rho$ and $\kappa^x\otimes\tau_i^x$ if and only if 
   \[
   \sum_{i\in I}\left( S_i\circ\kappa(\mathfrak{B}_{\rm max})(u)\right)\otimes\left( T_i\circ\rho (u)\right) = \sum_{i\in I}\left(\kappa(\mathfrak{B}_{\rm max})^x(u)\circ S_i\right)\otimes\left(\tau_i^x (u)\circ T_i\right)
   \]
   for all $u\in U(\mathfrak{B}_0)J^1(\mathfrak{B}_{\rm max})\cap U(\mathfrak{B})^xJ^1(\mathfrak{B}_{\rm max})^x$. In particular, if $\varphi$ intertwines these representations, for $u\in J^1(\mathfrak{B}_{\rm max})\cap J^1(\mathfrak{B}_{\rm max})^x$, we must have 
   \[
   \sum_{i\in I}\left( S_i\circ \kappa(\mathfrak{B}_{\rm max})(u) - \kappa(\mathfrak{B}_{\rm max})^x(u)\circ S_i  \right)\otimes T_i =0.
   \]
   Since the $T_i$ are linearly independent, we obtain
   \[
   S_i\in \Hom_{J^1(\mathfrak{B}_{\rm max})\cap J^1(\mathfrak{B}_{\rm max})^x}(\kappa(\mathfrak{B}_{\rm max}), \kappa(\mathfrak{B}_{\rm max})^x).
   \]
   By \cite[Proposition 5.2.7]{bushnell1993admissible}, the latter space is 1-dimensional. It follows that any $\varphi\in\Hom_{J_{\rm max}\cap U(\mathfrak{B})^xJ^1(\mathfrak{B}_{\rm max})^x}(\lambda_{\rm max}, \kappa^x\otimes\tau_i^x)$ is of the form $S\otimes T$, where \[
   S\in\Hom_{J(\mathfrak{B}_{\rm max})\cap J(\mathfrak{B}_{\rm max})^x}(\kappa(\mathfrak{B}_{\rm max}), \kappa(\mathfrak{B}_{\rm max})^x)
   \]
   and $T\in\Hom_{\mathbb{C}}(X_0,X)$. We also note that in this case we necessarily have
   \[
   T\in\Hom_{U(\mathfrak{B}_0)\cap U(\mathfrak{B})^x}(\rho,\tau_i^x).
   \]
   In this way, we obtain an isomorphism of $\mathbb{C}$-vector spaces 
   \begin{multline*}
   \Hom_{J_{\rm max}\cap U(\mathfrak{B})^xJ^1(\mathfrak{B}_{\rm max})^x}(\lambda_{\rm max}, \kappa^x\otimes\tau_i^x) \\
   \xrightarrow{\ \sim \ }\Hom_{J^1(\mathfrak{B}_{\rm max})\cap J^1(\mathfrak{B}_{\rm max})^x}(\kappa(\mathfrak{B}_{\rm max}), \kappa(\mathfrak{B}_{\rm max})^x)\otimes \Hom_{U(\mathfrak{B}_0)\cap U(\mathfrak{B})^x}(\rho, \tau_i^x).
   \end{multline*}
   Since the space $\Hom_{J^1(\mathfrak{B}_{\rm max})\cap J^1(\mathfrak{B}_{\rm max})^x}(\kappa(\mathfrak{B}_{\rm max}), \kappa(\mathfrak{B}_{\rm max})^x)$ is 1-dimensional, Lemma \ref{ss:lemma:hom:2} follows.
\end{proof}

\subsection{} We are now in a position to state and prove the conjecture of Broussous and Schneider. Let $S_\Sigma\colon C_q(\Sigma,\mathcal{C})\to\mathcal{V}$ be the $J_{\rm{max}}$-homomorphism defined by 
\begin{equation*}
    S_\sigma(\omega)=\sum_{\sigma\in\Sigma}\omega(\sigma),
\end{equation*}
yielding the exact sequence of $J_{\rm max}$-modules 
\begin{equation*}
    0\longrightarrow K_{\Sigma}\longrightarrow C_q(\Sigma,\mathcal{C})\longrightarrow \sum_{j\in J_{\rm max}}jx\mathcal{V}[\sigma_{\mathfrak{C}}]\longrightarrow 0,
\end{equation*}
where $K_\Sigma=\ker S_\Sigma$. Taking $\lambda_{\rm max}$-isotypic components, we obtain the exact sequence 
\begin{equation*}
    0\longrightarrow K^{\lambda_{\rm max}}_{\Sigma}\longrightarrow C_q(\Sigma,\mathcal{C})^{\lambda_{\rm max}}\longrightarrow \left(\sum_{j\in J_{\rm max}}jx\mathcal{V}[\sigma_{\mathfrak{C}}]\right) ^{\lambda_{\rm max}}\longrightarrow 0.
\end{equation*}
 As given in the equation preceding \cite[Conjecture X.4.1]{BS2017type},
 \begin{equation*}
    \left(\sum_{j\in J_{\rm max}}jx\mathcal{V}[\sigma_{\mathfrak{C}}]\right) ^{\lambda_{\rm max}}=\left\lbrace \begin{array}{ll}
       \mathcal{V}^{\lambda_{\rm max}} & {\rm if }\ x\in J_{\rm max}G_LU(\mathfrak{A}),\\
         0 & {\rm otherwise}.
    \end{array}\right.
\end{equation*}

\begin{Theorem}[\emph{(Broussous-Schneider's Conjecture)}]\label{main:thm}
    \emph{Let $q \leq d$ be a non-negative integer, and let $\Sigma = J_{\rm max} x \cdot \sigma_{\mathfrak{C}}$ be a $J_{\rm max}$-orbit in $X[L]_q$, where $\mathfrak{C}$ is an $\mathfrak{o}_L$-order in $C$ with $\mathfrak{C}_{\rm min} \subseteq \mathfrak{C} \subseteq \mathfrak{C}_{\rm max}$ and $x \in G$. Then \( K_\Sigma^{\lambda_{\rm max}} = 0 \). Equivalently,
    \[
        \dim \Hom_{J_{\rm max}}(\lambda_{\rm max}, C_q(\Sigma, \mathcal{C})) = \dim \Hom_{J_{\rm max}}(\lambda_{\rm max}, \mathcal{V}^{\lambda_{\rm max}}),
    \]
    for all \( x \in J_{\rm max} G_L U(\mathfrak{A}) \), whereas
    \[
        \dim \Hom_{J_{\rm max}}(\lambda_{\rm max}, C_q(\Sigma, \mathcal{C})) = 0
    \]
    for \( x \not\in J_{\rm max} G_L U(\mathfrak{A}) \).
    }
\end{Theorem}

\begin{proof}
Proceeding as in the proof of Lemma \ref{ss:lemma:hom:1}, we arrive at
\begin{equation*}
    \Hom_{J_{\rm max}}( \lambda_{\rm max}, C_q(\Sigma,\mathcal{C})) = \bigoplus_{i=1}^m\bigoplus_{u\in U}\Hom_{J_{\rm max}\cap U(\mathfrak{B})^{ux}J^1(\mathfrak{B}_{\rm max})^{ux}}\left(\lambda_{\rm max}, \kappa(\mathfrak{B}_{\rm max})^{ux}\otimes \tau_i^{ux}\right),
\end{equation*}
where $U$ is the double coset set
\begin{equation*}
    U=J_{\rm max}\cap U(\mathfrak{A})^x\backslash U(\mathfrak{A})^x/U(\mathfrak{A})^xJ^1(\mathfrak{B}_{\rm max})^x.
\end{equation*}
Furthermore, we have seen that if
\begin{equation}\label{eq:hom:not:0}
    \Hom_{J_{\rm max}\cap U(\mathfrak{B})^{ux}J^1(\mathfrak{B}_{\rm max})^{ux}}\left(\lambda_{\rm max}, \kappa(\mathfrak{B}_{\rm max})^{ux}\otimes \tau_i^{ux}\right)\not= 0
\end{equation}
for some $u\in U(\mathfrak{B})^x$, then $uxU(\mathfrak{B})J^1(\mathfrak{B}_{\rm max})\cap J_{\rm max}G_LJ_{\rm max}\not = \emptyset$. Since $uxU(\mathfrak{B})J^1(\mathfrak{B}_{\rm max})\subseteq x U(\mathfrak{B})$ for all $u\in U(\mathfrak{A})^x$, then \eqref{eq:hom:not:0} implies that 
\[
xU(\mathfrak{A})\cap J_{\rm max}G_LJ_{\rm max}\not= \emptyset,
\]
that is $x\in J_{\rm max}G_LU(\mathfrak{A})$. Thus, Theorem \ref{main:thm} holds when $x\not\in J_{\rm max}G_LU(\mathfrak{A})$. 

Let us assume that $x\in J_{\rm max}G_LU(\mathfrak{A})$. Writing $j=jg_Lu$, $j\in J_{\rm max}$, $g_L\in G_L$ and $u\in U(\mathfrak{A})$, we have that 
\[
\Sigma=J_{\rm max}x\sigma_{\mathfrak{C}}=J_{\rm max}g_L\sigma_{\mathfrak{C}},
\]
so we may and do assume that $x\in G_L$. By Lemma \ref{ss:lemma:hom:1} and Lemma \ref{ss:lemma:hom:2}, we have 
\begin{equation}\label{eq:sum:dim:lemm}
\dim \Hom_{J_{\rm max}}(\lambda_{\rm max},C_q(\Sigma,\mathcal{C}))=\sum_{i=1}^m \dim\Hom_{U(\mathfrak{B}_0)\cap U(\mathfrak{B})^x}(\rho,\tau_i^x).
\end{equation}
Since $\lambda_{\rm max}$ is $\eta(\mathfrak{B},\mathfrak{B}_{\rm max})$-isotypic, then 
\begin{equation*}
    \Hom_{J_{\rm max}}(\lambda_{\rm max},\mathcal{V})=\Hom_{J_{\rm max}}(\lambda_{\rm max},\mathcal{V}^{\eta(\mathfrak{B},\mathfrak{B}_{\rm max})})=\bigoplus_{i=1}^m\Hom_{J_{\rm max}}(\lambda_{\rm max}, \kappa(\mathfrak{B}_{\rm max})\otimes \tau_i).
\end{equation*}
Thus, Theorem \ref{main:thm} holds for $x=1$, and it remains to prove that \eqref{eq:sum:dim:lemm} is dependent of $x$, provided $x\in G_L$. This follows from Lemma \ref{lem:techn}.
\end{proof}

\subsection{}\label{ss:main}  As a consequence of Theorem \ref{main:thm}, we obtain the following result, as already observed in \cite[Theorem X.5.2]{BS2017type}, which is now unconditional. 

\begin{thm*}
     Let $(J,\lambda)$ be a simple type and let $\mathcal{V}$ be any admissible representation in $\mathcal{R}_{(J,\lambda)}(G)$. The augmented chain complex
     \begin{equation}\label{eq:bs:res}
    C_\bullet^{{\rm or}}(X[L], \mathcal{C}[\mathcal{V}])\longrightarrow \mathcal{V}
    \end{equation}
    is a resolution of $\mathcal{V}$ on $\mathcal{R}_{(J,\lambda)}(G)$. If $\mathcal{V}$ has central character $\chi$, then the resolution is projective in $\mathcal{R}_\chi(G)$.
\end{thm*}

\bibliographystyle{acm}
\bibliography{Ref.bib} 

\medskip

\noindent{\sc \Small Francisco Javier Navarro, Instituto de Matem\'aticas, Pontificia Universidad Cat\'olica de Valpara\'iso, Blanco Viel 596, Cerro Bar\'on, Valpara\'iso, Chile}\\
\emph{\Small E-mail address: }\texttt{\Small \href{mailto:francisco.navarro.n@pucv.cl}{francisco.navarro.n@pucv.cl}}

\end{document}